%% file: approximating_nonstationary_Markov_chains.tex
\documentclass[11pt,twoside]{article}

\usepackage{geometry}
\usepackage{setspace}
\usepackage{fancyhdr}
\usepackage{hyperref}

\usepackage{amsmath,amsthm,amssymb,amsfonts}
\usepackage{mathtools}
\usepackage[mathscr]{euscript}
\usepackage{bm}
\usepackage{undertilde}
\usepackage{subcaption}
\usepackage{color}
\newtheorem{assumption}{Assumption}
\newtheorem{proposition}{Proposition}

\newtheorem{theorem}{Theorem}
\theoremstyle{remark}

\makeatletter
\newtheoremstyle{mytheoremstyle}%
{\topsep}{\topsep}
{}{}            
{}{}            
{0.5em}
{\thmname{\@ifempty{#3}{#1}\@ifnotempty{#3}{#3}}}
\makeatother
\theoremstyle{mytheoremstyle}

\allowdisplaybreaks

\usepackage{enumitem} 
\usepackage{caption} 
\usepackage{booktabs} 
\newsavebox\CBox

\usepackage{natbib}
 \bibpunct[, ]{(}{)}{,}{a}{}{,}%
 \def\newblock{\ }%
\makeatletter
\newcommand{\vast}{\bBigg@{4}}
\newcommand{\Vast}{\bBigg@{5}}   
\makeatother

\def\E{\mathbb{E}}

\begin{document}

\title{\vspace{-1em}
Approximating Performance Measures for Slowly Changing Non-stationary Markov Chains \vspace{-0.6em}}

\author{
{\normalsize Zeyu Zheng} \vspace{-0.6em}\\
{\footnotesize Department of Management Science and Engineering} \vspace{-0.8em}\\
{\footnotesize Stanford University, zyzheng@stanford.edu}\\
{\normalsize Harsha Honnappa} \vspace{-0.6em}\\
{\footnotesize School of Industrial Engineering} \vspace{-0.8em}\\
{\footnotesize Purdue University, honnappa@purdue.edu}\\
{\normalsize Peter W. Glynn} \vspace{-0.6em}\\
{\footnotesize Department of Management Science and Engineering} \vspace{-0.8em}\\
{\footnotesize Stanford University, glynn@stanford.edu}\\
}

\date{\today}

\maketitle

\pagestyle{fancy}
\fancyhead[RO,LE]{\small \thepage}
\fancyfoot[C]{}

\begin{abstract} 
\onehalfspacing
\noindent
 This paper is concerned with the development of rigorous approximations to various expectations associated with Markov chains and processes having non-stationary transition probabilities. Such non-stationary models arise naturally in contexts in which time-of-day effects or seasonality effects need to be incorporated. Our approximations are valid asymptotically in regimes in which the transition probabilities change slowly over time. Specifically, we develop approximations for the expected infinite horizon discounted reward, the expected reward to the hitting time of a set, the expected reward associated with the state occupied by the chain at time $n$, and the expected cumulative reward over an interval $[0,n]$. In each case, the approximation involves a linear system of equations identical in form to that which one would need to solve to compute the corresponding quantity for a Markov model having stationary transition probabilities. In that sense, the theory provides an approximation no harder to compute than in the traditional stationary context. While most of the theory is developed for finite state Markov chains, we also provide generalizations to continuous state Markov chains, and finite state Markov jump processes in continuous time. In the latter context, one of our approximations coincides with the uniform acceleration asymptotic due to \cite{massey1998uniform}. \\

\vspace{-1em}
\end{abstract}

\vspace{-1em}
\noindent\makebox[\linewidth]{\rule{\linewidth}{0.8pt}}

\newpage

\section{Introduction}

The realistic modeling of many problems arising in operations research and operations management requires explicit incorporation of time-of-day effects, day-of-week effects, or seasonality effects. Such non-stationarities also occur in settings in which secular trends, such as a steady increase in demand for a product, may affect the system dynamics over an operationally meaningful time scale. Unfortunately, the development of closed-form theory for the great majority of stochastic models requires an assumption of stationary dynamics, free of such non-stationarities. In particular, in the Markov chain and continuous time Markov process setting, the assumption of stationary transition probabilities gives one the ability to easily compute many performance measures as solutions to systems of linear equations; see, for example, \cite{asmussen2008applied} and \cite{heyman1982stochastic}. Specifically, the equilibrium distribution of such processes can easily be computed by solving the linear system corresponding to the stationarity distribution. 

Given this state of affairs, it would be convenient if there were a general approach to obtaining approximations for Markov models having non-stationary transition probabilities in which the approximations involved linear systems of equations of identical structure to that obtained in the setting of stationary transition probabilities. For example, the approximating linear systems for non-stationary birth-death processes should ideally be tri-diagonal, just as are the linear systems arising in the setting of stationary birth-death processes. As far as we are aware, there is no such general approach presently available. In this paper, we provide such a set of approximations. Our approximations are valid precisely in the context in which one would presume that an approximation by a model with stationary dynamics should be valid, namely a regime in which the transition probabilities change slowly over time (or are ``slowly changing"). 

In this slowly changing setting, we show how various performance measures for Markov chains $X = (X_{n}: n\ge0)$ with non-stationary transition probabilities can be approximated by corresponding calculations involving stationary transition probabilities. In Section 2, we start by illustrating this approach in the setting of expected infinite horizon discounted reward for finite state Markov chains. Section 3 extends this to calculations involving the expected cumulative reward to the hitting time of a set, such as those arising in dependability modeling and actuarial risk calculation. Section 4 is concerned with two different approximations for transient expectations of the form $\E r(X_{n})$, one involving ``Taylor expanding" in terms of the transition matrix associated with the first step over the horizon $[0,n]$, while the second expands in terms of the last step $n$ associated with the horizon $[0,n]$. The second approach coincides, when specialized to the uniform acceleration (UA) asymptotic regime, to a discrete-time version of the first term in the UA asymptotic expansion of \cite{massey1998uniform}; see also \cite{khashinskii1996asymptotic}. This section also shows how the theory extends to continuous state space discrete time uniformly ergodic Markov chains. 

In Section 5, we develop an approximation for expected cumulative reward over $[0,n]$, that takes into account the influence of the initial distribution. Our last theory section, Section 6, is concerned with showing how the theory extends to Markov jump processes, and derives an approximation for $\E r(X(t))$, when $X = (X(t): t\ge0)$ is a finite state Markov jump process with a slowly changing family of rate matrices $(Q(t):t\ge 0)$. When specialized to the UA setting, this coincides precisely with the UA approximation of \cite{khashinskii1996asymptotic} and \cite{massey1998uniform}. However, our derivation makes clearer that the specific time acceleration associated with the UA expansion is not required for the approximation to be valid. Rather, the only requirement needed for the validity of this approximation is the slow variation of order $\epsilon$ in the rate matrices $Q(s)$ for value of $s$ within $(\log(1/\epsilon))^2$ of $t$. In particular, the rate matrices can vary rapidly prior to such times $s$ without affecting the validity of the approximation. Section 7 concludes the paper with a brief numerical study of these approximations. Throughout the paper, we make an effort to relate the asymptotic theory, involving a parameter $\epsilon$ being sent to 0, to modeling settings in which a given Markov chain (with no asymptotic parameter) needs to be approximated, and make specific proposals for how the approximation can be implemented. 

A distinguishing characteristic of this work is the development of an approximation theory for generic un-structured Markov chains and processes. In the present of specific models, one can develop model-specific approximations that can provide effective numerical and analytical approximations to non-stationary versions of such models. For instance, the work done by \cite{massey1985asymptotic} for the non-stationary $M/M/1$ queue is in this spirit, as is the limit theory on non-stationary reflected Brownian motion developed by \cite{mandelbaum1995strong}. The work done by \cite{whitt1991pointwise} establishes the asymptotic correctness of the pointwise stationary approximation for $M_t/M_t/s$ queues by assuming local stationarity. There is also a substantial literature on exact analysis of infinite-server queues that establishes that non-stationarity does not seriously complicate such models relative to the stationary case. Finally, there is a significant and growing body of contributions on closure approximations for queues that shows promise of generating efficient numerical algorithms for analysis of such systems; see \cite{massey2013gaussian}, \cite{pender2014gram,pender2014poisson,pender2015nonstationary}. A recent survey of this work is provided by \cite{whitt2017queues}. As noted above, the current paper has a different focus, namely that of developing approximations for generic Markov models.

\section{Approximating Expected Infinite Horizon Discounted Reward}
Let $X=(X_k:k \ge 0)$ be a finite state $S$-valued Markov chain. For each $x\in S$, suppose that $r(x)$ is the reward obtained for spending one unit of time in state $x$. In this section, we are concerned with the expected infinite horizon discounted reward defined by 
\begin{equation}
\kappa = \E \sum_{j=0}^{\infty}e^{-\alpha j}r(X_j)
\end{equation}
for a given (per period) discount rate $\alpha>0$. The Markov property implies the existence of a sequence $(P_k:k\ge 1)$ of stochastic matrices for which 
\[
P(X_{k+1}=y\mid X_0,\ldots,X_k) = P_{k+1}(X_k,y) \quad a.s.
\]
for $k\ge 0$. Put $\mu(x) = P(X_0=x)$ for $x\in S$. We adopt the convention that all probability mass functions on $S$ are encoded as row vectors, and all real-valued functions with domain $S$ are encoded as column vectors. We further adopt the convention that in using the product notation $\prod_{i=k}^{m} A_i$ for a product of square matrices $A_k,\cdots,A_m$, we always multiply them in increasing order of their indices or time arguments, so that (for example) $\prod_{i=k}^{m} A_i = A_k\cdots A_m$ and $\prod_{i=k}^{m} A_{n-i} = A_{n-m}\cdots A_{n-k}$, and a product over an empty set of indices equals 1.  With this convention in hand, it is evident that 
\begin{equation}\label{eq:2.2}
\kappa = \sum_{j=0}^{\infty}e^{-\alpha j}\mu\left(\prod_{k=1}^{j}P_k\right)r.
\end{equation}

It is well known that if $X$ has stationary transition probabilities (so that $P_k = P_1$ for $k\ge 1$), then $\kappa = \mu \nu$, where $\nu$ satisfies 
\begin{equation}\label{eq:2.3}
\nu = r+ e^{-\alpha}P_1\nu.
\end{equation}
The only finite-valued solution of (\ref{eq:2.3}) is then given by $\nu = 	\sum_{j=0}^{\infty}e^{-\alpha j}P_1^j = (I-e^{-\alpha}P_1)^{-1}r$; see \cite{kemeny1960finite}. Given that (\ref{eq:2.3}) can be solved by Gaussian elimination, $\nu$ (and hence $\kappa$) can be computed in $O(|S|^3)$ arithmetic (or floating point) operations; see \cite{farebrother1988linear}. (Here, we use the notation $O(g(|S|,\theta,n,\alpha,\epsilon))$) to denote a function that is bounded by a multiple of $g(|S|,\theta,n,\alpha,\epsilon)$ and and the notation $o(g(\epsilon))$ to denote a function for which $o(g(\epsilon))/g(\epsilon)\rightarrow0$ as $\epsilon\downarrow0$.) Alternatively, one can sometimes analytically calculate the solution to (\ref{eq:2.3}) in closed form. For example, such a closed form is always available for birth-death chains, because the linear system (\ref{eq:2.3}) is then tri-diagonal (so that the only non-zeros entries in the coefficient matrix appear on the diagonal, super-diagonal, and sub-diagonal).

We will now show how one can compute approximations to (\ref{eq:2.2}) in the non-stationary setting. These approximations involve linear systems of equations with coefficient matrices identical to those arising in (\ref{eq:2.3}). Consequently, these approximations are typically no harder to compute, either analytically or numerically, than are the linear systems associated with the stationary case. 

In preparation for stating our main result, we define the following norms on row vectors $\eta$, column vectors $f$, and square matrices $A$: 
\begin{align*}
\|\eta\| &= \sum_x |\eta(x)|;\\
\|f\| &= \max_x |f(x)|;\\
\|A\| &= \max_x \sum_y |A(x,y)|.
\end{align*}
We recall that $|\eta f| \le \|\eta\| \cdot \|f\|$, $\|\eta A\| \le \|\eta\| \cdot \|A\|$, $\|Af\| \le \|A\|\cdot \|f\|$, and $\|A_1A_2\|\le \|A_1\|\cdot \|A_2\|$ for matrices $A_1$ and $A_2$; see, for example, \cite{golub2012matrix}. We will need the following basic result. 
\begin{proposition}\label{prop:1}
	Suppose that $\| A^n \| < 1$ for some $n\ge 1$. Then:
	\begin{itemize}
		\item[i.)] $\|A^n \|\rightarrow 0$ geometrically fast as $n\rightarrow\infty$, and $\sum_{n=0}^{\infty} A^n = (I-A)^{-1}$;
		\item[ii.)]for $j\ge 1$, 
		\[
		\sum_{n=0}^{\infty}(n+j)(n+j-1)\cdots(n+1) A^n = j!(I-A)^{-j-1}.
		\]
	\end{itemize}
\end{proposition}

\begin{proof}{Proof of Proposition \ref{prop:1}.}
	Part $i.)$ is well known; see \cite{kemeny1960finite}. Part $ii.)$ is yielded by noting that $(I-A)^{-j-1} = (\sum_{n=0}^{\infty}A^n)^{j+1}$ is a Binomial series with a negative exponent and equals $\sum_{n=0}^{\infty} {n+j\choose j} A^n $.
\end{proof}



Because we wish to develop asymptotic approximations, we now consider a parameterized family of infinite horizon discounted rewards defined by 
\[
\kappa(\epsilon) = \sum_{j=0}^{\infty}e^{-\alpha j}\mu \left(\prod_{k=1}^{j}P_k(\epsilon)\right) r,
\]
where the $P_k(\epsilon)$'s are stochastic matrices. We now make the following assumption about the family $(P_k(\epsilon): k \ge 1)$.

\begin{assumption}~\label{assume:1}
	Suppose that $(P_k(\epsilon): k \ge1 ,\epsilon >0)$ is a family of stochastic matrices for which there exist scalars $(a_{i1}:i\ge 1)$, matrices $\tilde{P}$, $\tilde{P}^{(1)}$, and scalars $s$, $\delta$, and $p>0$ such that
	\begin{itemize}
		\item[i)] $\sup_{1\le i\le (\log(1/\epsilon))^{1+\delta}} \| P_i(\epsilon) - (\tilde{P}+\epsilon a_{i1}\tilde{P}^{(1)}) \| = O(\epsilon^2 (\log(1/\epsilon))^s)$  as $\epsilon\downarrow 0 $;
		\item[ii)] $|a_{i1}| = O(i^p)$ as $i\rightarrow\infty$. 
	\end{itemize}
\end{assumption}

\begin{theorem}\label{thm:1}
	Suppose that Assumption \ref{assume:1} holds. Then, there exists $w<\infty$ such that 
	\begin{equation}
	\kappa(\epsilon) = \mu(I-e^{-\alpha}\tilde{{P}})^{-1}r + \epsilon e^{-\alpha} \mu \sum_{k=0}^{\infty} a_{k+1,1} e^{-\alpha k} \tilde{P}^k \tilde{P}^{(1)} (I-e^{-\alpha}\tilde{P})^{-1}r + O(\epsilon^2 (\log(1/\epsilon))^w)\label{eq:1new}
	\end{equation}
	as $\epsilon\downarrow 0 $.
\end{theorem}

We note that Assumption \ref{assume:1} asserts that the $P_k(\epsilon)$'s are ``slowly changing" over a time scale of order $(\log(1/\epsilon))^{1+\delta}$ (so that we require slow variation only over a small portion of the entire horizon). In the presence of such an assumption, Theorem 1 provides a first-order correction to the stationary formula $\mu (I-e^{-\alpha}\tilde{P})^{-1}r$ that reflects the non-stationary dynamics of the Markov chain. 

One possible choice for $(P_k(\epsilon):k \ge 1, \epsilon >0)$ is one in which $P_k(\epsilon) = P(\epsilon)$, where $P(\cdot)$ is twice continuously differentiable in a neighborhood of 0. In this case, Assumption \ref{assume:1} holds with $a_{i1} = 1$, $\tilde{P} = P(0)$, $\tilde{P}^{(1)} = P^{(1)}(0)$ (when $P^{(j)}(\theta)$ is the $j$'th derivative of $P(\cdot)$ evaluated at $\theta$), and $s=0$, in which case 
\begin{align*}
\kappa(\epsilon) = \mu (I-e^{-\alpha}\tilde{P})^{-1}r + \epsilon e^{-\alpha}\mu(I-e^{-\alpha}\tilde{P})^{-1}\tilde{P}^{(1)}(I-e^{-\alpha}\tilde{P})^{-1}r + O(\epsilon^2(\log(1/\epsilon))^w)
\end{align*}
as $w\downarrow 0$. Of course, in this setting, the $P_k(\epsilon)$'s describe a Markov chain with stationary transition probabilities, so our formula (\ref{eq:1new}) is then just computing the sensitivity of $\kappa(\epsilon)$ to a perturbation in the common one-step transition matrix. 

A genuinely non-stationary example arises when $P_k(\epsilon) = P((k-1)\epsilon)$, where $P(\cdot)$ is again twice continuously differentiable in a neighborhood of the origin. (Note that Assumption \ref{assume:1} governs only the first $(\log(1/\epsilon))^{1+\delta}$) transitions, thereby corresponding to the behavior of $P(\cdot)$ over $[0,\epsilon(\log(1/\epsilon))^{1+\delta}]$, so that only $P(\cdot)$'s behavior near 0 plays a role.) In this case, $a_{i1}=(i-1)$, $\tilde{P} = P(0)$, $\tilde{P}^{(1)} = P^{(1)}(0)$, and $s = 2(1+\delta)$. In view of part ii) of Proposition 1 with $j = 1$, formula (\ref{eq:1new}) then takes the form 
\begin{equation}
\kappa(\epsilon) = \mu(I-e^{-\alpha}\tilde{P})^{-1}r + \epsilon e^{-2\alpha}\mu \tilde{P}(I-e^{-\alpha}\tilde{P})^{-2}\tilde{P}^{(1)}(I-e^{-\alpha}\tilde{P})^{-1}r + O(\epsilon^2(\log(1/\epsilon))^w) \label{eq:1anew}
\end{equation} 
as as $\epsilon\downarrow 0 $. Thus, with this parametrization, the infinite sum in the first-order approximation to $\kappa(\epsilon)$ that accounts for the non-stationarity collapses to a quantity involving $(I-e^{-\alpha}\tilde{P})^{-1}.$ In fact, 
\[
\kappa(\epsilon) = \mu \nu + \epsilon e^{-2\alpha} \mu\tilde{P}\nu_2 + O(\epsilon^2 (\log(1/\epsilon))^w)
\]
as $\epsilon\downarrow 0$, where $\nu$, $\nu_1$, and $\nu_2$ satisfy the linear systems 
\begin{align*}
&(I-e^{-\alpha}\tilde{P})\nu  = r,\\
&(I-e^{-\alpha}\tilde{P})\nu_1 = \tilde{P}^{(1)}\nu,\\
&(I-e^{-\alpha}\tilde{P})\nu_2 = \nu_1.
\end{align*}
As proposed earlier, the coefficient matrices are all identical and equal to the coefficient matrix associated with computing infinite horizon discounted reward in the stationary dynamics setting (where $P_k(\epsilon) \equiv \tilde{P}$).

We now turn to the proof of Theorem 1. The key estimate is provided by the next result. 

\begin{proposition}\label{prop:2}
	Under Assumption \ref{assume:1}, there exists $w>0$ such that 
	\begin{equation*}
	\sup_{1\le m\le (\log(1/\epsilon))^{1+\delta}} \left\| P_1(\epsilon)\cdots P_m(\epsilon) - \left(\tilde{P}^m+\epsilon \sum_{i=1}^{m} a _{i1}\tilde{P}^{i-1}\tilde{P}^{(1)}\tilde{P}^{m-i}\right) \right\| = O(\epsilon^2(\log(1/\epsilon))^w)
	\end{equation*}
	as $\epsilon\downarrow0.$
\end{proposition}

\begin{proof}{Proof of Proposition \ref{prop:2}.}
For $\epsilon$ sufficiently small, there exists $d<\infty$ such that
\begin{equation}
\| P_i(\epsilon) - (\tilde{P}+\epsilon a_{i1}\tilde{P}^{(1)}) \| \le d \epsilon^2 (\log(1/\epsilon))^s \label{eq:Ptildedistance}
\end{equation}
and 
\[
|a_{i1} |\le d\, i ^p
\]
for $1\le i \le (\log(1/\epsilon))^{1+\delta}$. Note that $\tilde{P}$ must be stochastic, as can be seen from the fact that the $P_i(\epsilon)$'s are stochastic and sending $\epsilon$ to 0 in (\ref{eq:Ptildedistance}). Recalling that $\|K\| = 1$ for a stochastic matrix and the fact that $\|K_1 K_2\| \le \|K_1\|\|K_2\|$ for arbitrary matrices $K_1$ and $K_2$, it is evident that 
\begin{align}
f_{m+1} & \triangleq \left\| \prod_{i=1}^{m+1}P_i(\epsilon) - \left( \tilde{P}^{m+1}+\epsilon \sum_{i=1}^{m+1}a_{i1}\tilde{P}^{i-1}\tilde{P}^{(1)}\tilde{P}^{m+1-i} \right) \right\|\nonumber\\
&\le \left\|  \left(\prod_{i=1}^{m}P_i(\epsilon) - \left( \tilde{P}^{m}+\epsilon \sum_{i=1}^{m}a_{i1}\tilde{P}^{i-1}\tilde{P}^{(1)}\tilde{P}^{m-i} \right)\right) \tilde{P} \right\| \nonumber \\
& + \left\| \prod_{i=1}^{m}P_i(\epsilon)  \left( {P}_{m+1}(\epsilon) - (\tilde{P} + \epsilon a_{m+1,1}\tilde{P}^{(1)}) \right) \right\| \nonumber\\
& + \left\|\left(\tilde{P}^m + \epsilon \sum_{i=1}^{m} a_{i1} \tilde{P}^{i-1}\tilde{P}^{(1)}\tilde{P}^{m-i} - \prod_{i=1}^{m}P_i(\epsilon) \right)\epsilon a_{m+1,1}\tilde{P}^{(1)} \right\|\nonumber\\
& + \left\| \left(\epsilon \sum_{i=1}^{m}a_{i1} \tilde{P}^{i-1}\tilde{P}^{(1)}\tilde{P}^{m-i}\right)\epsilon a_{m+1,1} \tilde{P}^{(1)} \right\| \nonumber\\
&\le f_m + \epsilon^2 d(\log(1/\epsilon))^s + \epsilon f_m |a_{m+1,1}| \|\tilde{P}^{(1)}\| + \epsilon^2 |a_{m+1,1}|\sum_{i=1}^{m}|a_{i1}| \|\tilde{P}^{(1)}\|^2 \label{eq:A5}
\end{align}
for $1\le m+1\le (\log(1/\epsilon))^{1+\delta}$. But $|a_{m+1,1}| \le d(m+1)^p$, and $\sum_{i=1}^{m}|a_{i1}|\le d'(m+1)^{p+1} \le d' (\log(1/\epsilon))^{(p+1)(1+\delta)}$ for some $d'<\infty$ when $m+1 \le (\log(1/\epsilon))^{1+\delta}$. So, 
\begin{align*}
f_{m+1} &\le f_m(1+ \epsilon d(\log(1/\epsilon))^{p(1+\delta)}\|\tilde{P}^{(1)}\|) + \epsilon^2\left[d(\log(1/\epsilon))^s + dd'(\log(1/\epsilon))^{p(1+\delta) + (p+1)(1+\delta)} \|\tilde{P}^{(1)}\|^2\right]\\
&\triangleq f_m c + d'',
\end{align*}
for which it follows that 
\begin{equation}
f_m \le d''(1+c^m) \le 2 d'' c^m \le 2 d'' c^{(\log(1/\epsilon))^{1+\delta}}
\label{eq:6.4}
\end{equation}
for $1\le m\le (\log(1/\epsilon))^{1+\delta}$. Note that $d'' = d''(\epsilon) = O(\epsilon^2(\log(1/\epsilon)))^{w}$ for some $w > 0$ as $\epsilon\downarrow0$, and $1\le c(\epsilon)\le 1+ \epsilon/2$ for $\epsilon$ sufficiently small, so that 
\begin{equation*}
(1+c(\epsilon))^{(\log(1/\epsilon))^{1+\delta}}\rightarrow 1
\end{equation*}
as $\epsilon\downarrow 0.$ Consequently, 
\[
f_m = f_m(\epsilon) = O(\epsilon^2(\log(1/\epsilon))^w),
\]
uniformly in $1\le m \le (\log(1/\epsilon))^{1+\delta}$, proving the result.
\end{proof}

\begin{proof}{Proof of Theorem \ref{thm:1}.} 
Note that 
\begin{align}
& \left\| \sum_{j >(\log(1/\epsilon))^{1+\delta}} e^{-\alpha j} \mu \prod_{k=1}^{j}P_k(\epsilon)r \right\| = O(\epsilon^2), \label{eq:2new}\\
& \left\| \sum_{j >(\log(1/\epsilon))^{1+\delta}} e^{-\alpha j} \mu \tilde{P}^jr \right \| = O(\epsilon^2), \label{eq:3new}\\
& \left\|\sum_{j >(\log(1/\epsilon))^{1+\delta}} e^{-\alpha j} \mu \prod_{k=1}^{j} a_{k1}\tilde{P}^{k-1}\tilde{P}^{(1)}\tilde{P}^{j-k} \right\| = O(\epsilon^2) \label{eq:4new}
\end{align}
as $\epsilon\downarrow 0$. On the other hand, Proposition 2 shows that 
\begin{align}
&\left\| \sum_{j\le (\log(1/\epsilon))^{1+\delta}} e^{-\alpha j}\mu \prod_{k=1}^{j}P_k(\epsilon)r - \sum_{j\le (\log(1/\epsilon))^{1+\delta}}e^{-\alpha j}\mu\left(\tilde{P}^j + \epsilon \sum_{k=1}^{j}a_{k1}\tilde{P}^{k-1}\tilde{P}^{(1)}\tilde{P}^{j-k}\right)r \right \| \nonumber\\
&\le \sum_{j\le (\log(1/\epsilon))^{1+\delta}} e^{-\alpha j} O(\epsilon^2(\log(1/\epsilon))^w) \nonumber\\
&= O(\epsilon^2(\log(1/\epsilon))^w) \label{eq:5new}.
\end{align}
as $\epsilon\downarrow 0.$ Relations (\ref{eq:2new}) through (\ref{eq:5new}) imply that 
\begin{align}
\kappa(\epsilon) =\,&  \mu\sum_{j=0}^{\infty} e^{-\alpha j} \tilde{P}^j r 	+ \epsilon \mu \sum_{j=1}^{\infty}e^{-\alpha j} \sum_{k=1}^{j} a_{k1}\tilde{P}^{k-1}\tilde{P}^{(1)}\tilde{P}^{j-k}r  + O(\epsilon^2(\log(1/\epsilon))^w) \label{eq:6new}
\end{align}
as $\epsilon\downarrow0$. But the second term on the right-hand side of (\ref{eq:6new}) equals
\[
\epsilon \mu \sum_{k=1}^{\infty} a_{k1}  \tilde{P}^{k-1} \tilde{P}^{(1)} e^{-\alpha k} \sum_{l=k}^{\infty}e^{-\alpha l} \tilde{P}^l r =   \epsilon \mu \sum_{k=1}^{\infty} a_{k1}  \tilde{P}^{k-1} \tilde{P}^{(1)}e^{-\alpha k} (I-e^{-\alpha}\tilde{P})^{-1}r,
\]
proving the theorem.
\end{proof}

We now discuss how the approximation can be applied in the setting of a Markov chain having dynamics governed by the sequence of transition matrices $(P_j: j\ge 1)$. If 
\begin{equation}
\max_{1\le j\le b/\alpha} \|P_j-P_1\| \label{eq:6anew}
\end{equation}
is small for some value of $b$ that is large, then (\ref{eq:1new}) suggests the approximation 
\[
\mu(I-e^{-\alpha}P_1)^{-1} + \sum_{j=1}^{\lfloor b/\alpha \rfloor} e^{-\alpha j} \sum_{k=1}^{j} P_1^{k-1}(P_k-P_1)P_1^{j-k},
\]
where $\lfloor x\rfloor$ denotes the floor of $x$. Of course, if 
\begin{equation}
P_j - P_1 = (j-1)\epsilon \tilde{P}^{(1)} \label{eq:7new}
\end{equation}
(as occurs when the $P_j$'s are consistent with the slowly changing smooth approximation $P_j = P(j\epsilon)$ for some value of $\epsilon$), then we have the more tractable approximation 
\begin{equation*}
\mu(I-e^{-\alpha}P_1)^{-1}r + e^{-2\alpha}\mu P_1(I-e^{-\alpha}P_1)^{-2}\epsilon \tilde{P}^{(1)} (I-e^{-\alpha}P_1)^{-1} r.
\end{equation*}

Note that $\tilde{P}^{(1)}$ can be approximated via any of the finite differences 
\begin{equation}
\epsilon \tilde{P}^{(1)} \approx\frac{P_j - P_1}{j-1} \label{eq:7anew}
\end{equation}
associated with (\ref{eq:7new}). Given the presence of the discount factor $e^{-\alpha},$ a reasonable choice for $j$ is likely to be something on the order of $(1-e^{-\alpha})^{-1}$.

There is no intrinsic difficulty in computing $j$'th order corrections, for any $j\ge 1$. To illustrate this point, we state the associated second-order correction.

\begin{assumption}~\label{assum:2}
	Suppose that $(P_k(\epsilon):k \ge 1, \epsilon > 0)$ is a family of stochastic matrices for which there exist scalars $((a_{i1},a_{i2}):i\ge 1)$, matrices $\tilde{P}$, $\tilde{P}^{(1)}$, $\tilde{P}^{(2)}$, and scalars $s,$ $\delta$, and $p>0$ such that
	\begin{itemize}
		\item[i)] $\sup_{1\le i\le (\log(1/\epsilon))^{1+\delta}} \| P_i(\epsilon) - (\tilde{P} + \epsilon a_{i1} \tilde{P}^{(1)} + \frac{\epsilon^2}{2}a_{i2}\tilde{P}^{(2)}) \| = O(\epsilon^3(\log(1/\epsilon))^s)$ as $\epsilon\downarrow 0$;
		\item[ii)] $|a_{ij}| = O(i^p)$ as $i\rightarrow\infty$, for $j=1,2$.
	\end{itemize}
	
\end{assumption}
In the presence of such slowly changing transition matrices over the logarithmic time scale $(\log(1/\epsilon))^{1+\delta}$, the following theorem is available. 

\begin{theorem}
	If Assumption \ref{assum:2} holds, then there exists $w>0$ such that 
	\begin{align}
	\kappa(\epsilon) =\,&  \mu(I-e^{-\alpha}\tilde{{P}})^{-1}r + \epsilon e^{-\alpha} \mu \sum_{k=0}^{\infty} a_{k+1,1} e^{-\alpha k} \tilde{P}^k \tilde{P}^{(1)} (I-e^{-\alpha}\tilde{P})^{-1}r \nonumber\\
	& +\frac{\epsilon^2}{2} e^{-\alpha} \mu \sum_{k=0}^{\infty} a_{k+1,2} e^{-\alpha k} \tilde{P}^k \tilde{P}^{(2)}(I-e^{-\alpha}\tilde{P})^{-1}r \nonumber\\
	&+ \epsilon^2 e^{-2\alpha} \mu \sum_{k=1}^{\infty} a_{k1} \tilde{P}^{k-1} e^{-\alpha(k-1)}\tilde{P}^{(1)} \sum_{l=k+1}^{\infty} a_{l1}\tilde{P}^{l-k-1}e^{-\alpha(l-k-1)}\tilde{P}^{(1)} (I-e^{-\alpha}\tilde{P})^{-1}r \nonumber\\
	&+ O(\epsilon^3 (\log(1/\epsilon))^w) \label{eq:8new}
	\end{align}
	as $\epsilon\downarrow 0.$
\end{theorem}

We omit the proof, since it is a direct extension of the argument used to establish Theorem 1. 

As in the setting of the first order correction, this formula greatly simplifies if we consider the case where $P_j(\epsilon) = P((j-1)\epsilon)$ for $j\ge 1$, where $P(\cdot)$ is three times continuously differentiable at $0$. In this case $a_{i1} = i-1$, $a_{i2} = (i-1)^2$, $\tilde{P} = P(0)$, $\tilde{P}^{(1) }= \tilde{P}^{(1)}(0)$, $\tilde{P}^{(2)} = P^{(2)}(0),$ and $s =3(1+\delta)$. Proposition 1 shows that the third term on the right-hand side of (\ref{eq:8new}) then equals 
\begin{align*}
& \frac{\epsilon^2}{2} e^{-\alpha} \mu \sum_{k=0}^{\infty} k^2 (e^{-\alpha}\tilde{P})^k \tilde{P}^{(2)}(I-e^{-\alpha }\tilde{P})^{-1} r \\
=\, & \frac{\epsilon^2}{2} e^{-2\alpha} \mu \tilde P \left[ 2 e^{-\alpha}\tilde P(I - e^{-\alpha} \tilde P)^{-3} +  (I - e^{-\alpha} \tilde P)^{-2}\right] \tilde P^{(2)}(I - e^{-\alpha} \tilde P)^{-1} r.
\end{align*}
Similarly, the fourth term equals 
\begin{align*}
& \epsilon^2 e^{-2\alpha} \mu \sum_{k=1}^{\infty}(k-1)(e^{-\alpha}\tilde{P})^{k-1}\tilde{P}^{(1)} \sum_{l=k+1}^{\infty} (l-1)(e^{-\alpha}\tilde{P})^{l-k-1}\tilde{P}^{(1)}(I-e^{-\alpha}\tilde{P})^{-1}r\\
=\,& 2 \epsilon^2 e^{-3\alpha} \mu \tilde P (I - e^{-\alpha} \tilde P)^{-3} \tilde P^{(1)} (I - e^{-\alpha} \tilde P)^{-1} \tilde P^{(1)} (I - e^{-\alpha} \tilde P)^{-1} r \\& +\epsilon^2 e^{-4\alpha} \mu \tilde P (I - e^{-\alpha} \tilde P)^{-2} \tilde P^{(1)} \tilde P (I - e^{-\alpha} \tilde P)^{-1} r,
\end{align*}
yielding the second-order approximation 
\begin{align}
\label{eq:secondterms}	\nonumber& \mu(I-e^{-\alpha}\tilde{P})^{-1}r +  \epsilon e^{-2\alpha} \mu \tilde{P}\left(I-e^{-\alpha}\tilde{P}\right)^{-2} \tilde{P}^{(1)}(I-e^{-\alpha}\tilde{P})^{-1}r\\
\nonumber   &+\frac{\epsilon^2}{2} e^{-2\alpha} \mu \tilde P \left[ 2 e^{-\alpha}\tilde P(I - e^{-\alpha} \tilde P)^{-3} +  (I - e^{-\alpha} \tilde P)^{-2}\right] \tilde P^{(2)}(I - e^{-\alpha} \tilde P)^{-1} r\\
\nonumber & + 2 \epsilon^2 e^{-3\alpha} \mu \tilde P (I - e^{-\alpha} \tilde P)^{-3} \tilde P^{(1)} (I - e^{-\alpha} \tilde P)^{-1} \tilde P^{(1)} (I - e^{-\alpha} \tilde P)^{-1} r \\& +\epsilon^2 e^{-4\alpha} \mu \tilde P (I - e^{-\alpha} \tilde P)^{-2} \tilde P^{(1)} \tilde P (I - e^{-\alpha} \tilde P)^{-1} r
\end{align}
for $\kappa.$

As with the first order approximation, the terms $\epsilon \tilde{P}^{(1)}$ and $\epsilon^2 \tilde{P}^{(2)}$ appearing in the approximation can be replaced by finite difference approximations, given by (\ref{eq:7anew}) and 
\begin{equation}
\epsilon^2 \tilde{P}^{(2)} \approx \frac{P_{2j-1}-2P_j + P_1}{(j-1)^2}, \label{eq:7aa}
\end{equation}
yielding an implementable approximation to $\kappa$ when condition (\ref{eq:6anew}) is in force.

\section{Approximating Expected Reward Cumulated to a Hitting Time} 
In this section, we extend the analysis of Section 2 to expectations of the form
\[
\delta = \E \sum_{j=0}^{T}r(X_j),
\]
where $T = \inf\{n\ge0: X_n\in C^c \}$ is the hitting time of a subset $C^c\subset S$. Such expectations arise in computing expected hitting times, and computing absorption probabilities of the form $P(X_T=y)$ (where $r(y)=1$ if $y\in C^c$ and 0 otherwise). These expectations are also of interest in dependability modeling and in actuarial risk calculations. 

For $i\ge 1$, let $B_i = (B_i(x,y):x,y\in C)$, where $B_i(x,y) = P_i(x,y)$ for $x,y\in C$, so that $B_i$ is the principal matrix of $P_i$ corresponding to ``$C$ to $C$" transitions. Then,
\begin{equation}
\delta = \sum_{j=0}^{\infty} \mu\left(\prod_{k=1}^{j}B_k\right)r_{j+1}, \label{eq:3.1}
\end{equation}
where $r_j(x) = r(x) + \sum_{y\in C^c} P_j(x,y)r(y) $ for $x\in C$.

Suppose that $X$ has stationary transition probabilities. If $(I-B_1)^{-1}$ exists, then $\delta = \mu w$, where $w=(I-B_1)^{-1}r_1$ and $w$ is the unique finite-valued solution of the linear system
\begin{equation}
w= r_1 + B_1 w. \label{eq:3.2}
\end{equation}
As in Section 2, our goal is to improve upon this zero'th order approximation $w$ to $\delta$, under the condition that $X$ has slowly changing transition probabilities. We adopt the framework of Section 2, and consider the function
\[
\delta(\epsilon) = \sum_{j=0}^{\infty} \mu \left(\prod_{k=1}^{j}{B}_k(\epsilon)\right)r_{j+1}(\epsilon),
\]
where $B_k(\epsilon)$ is the corresponding principal sub-matrix of $P_k(\epsilon)$ and $r_k(\epsilon) = (r_k(\epsilon,x):x\in C)$, with $r_k(\epsilon,x) = r(x) + \sum_{y\in C^c} P_k(\epsilon,x,y) r(y)$, for $k\ge 1$.
\begin{theorem}\label{thm:3}
	Assume that there exist $l\ge 1$, $s>0$, $\delta>0$, and matrices $\tilde{P}$ and $\tilde{P}^{(1)}$ such that:
	\begin{itemize}
		\item[i)] $\sup_{1\le i\le (\log(1/\epsilon))^{1+\delta}} \|P_i(\epsilon) - (\tilde{P} + \epsilon(i-1)\tilde{P}^{(1)})\| = O(\epsilon^2 (\log(1/\epsilon))^s)$ as $\epsilon\downarrow0;$
		\item[ii)] $\sup_{\substack{k\ge 0, \, \epsilon>0}} \|B_{k+1}(\epsilon)B_{k+2}(\epsilon)\cdots B_{k+l}(\epsilon)\| < 1.$
	\end{itemize}
	Then, there exists $w<\infty$ such that 
	\[
	\delta(\epsilon) = (I-\tilde{B})^{-1}r + \epsilon\mu \tilde{B}(I-\tilde{B})^{-2} \tilde{B}^{(1)} (I-\tilde{B})^{-1}\tilde{r} + \epsilon\mu \tilde{B} (I-\tilde{B})^{-2} \tilde{r}^{(1)} + O(\epsilon^2(\log(1/\epsilon))^w)
	\]
	as $\epsilon\downarrow0$, where $\tilde{r}(x) = r(x) + \sum_{y\in C^c}\tilde{P}(x,y)r(y)$, $\tilde{r}^{(1)}(x) = \sum_{y\in C^c} \tilde{P}^{(1)}(x,y)r(y)$, $\tilde{B}=(\tilde{P}(x,y):x,y\in C)$, and $\tilde{B}^{(1)} = (\tilde{P}^{(1)}(x,y):x,y\in C).$
\end{theorem}
\begin{proof}{Proof of Theorem \ref{thm:3}.}
Condition $ii)$ ensures that
\[
\| B_1(\epsilon)\cdots B_{ml}(\epsilon)  \| \le \beta^m,
\]
where $\beta\triangleq \sup\{ \|B_{k+1}(\epsilon) \cdots B_{k+l}(\epsilon)\|: k\ge0, \epsilon >0 \} <1$. Furthermore, 
\[
\|\tilde{B}^{l}\| = \limsup_{\epsilon\downarrow 0}\| B_1(\epsilon)\cdots B_l(\epsilon) \| \le \beta,
\]
from which it follows that
\begin{equation}
\left\| \sum_{j> (\log(1/\epsilon))^{1+\delta}} \mu\left(\prod_{k=1}^{j}B_k(\epsilon)\right)r_{j+1}(\epsilon)   \right\|= O(\epsilon^2), \label{eq:1sec3new}
\end{equation}
\begin{equation}
\left\| \sum_{j> (\log(1/\epsilon))^{1+\delta}} j\mu\tilde{B}^j \tilde{r}\,\right\| = O(\epsilon^2) \label{eq:2sec3new},
\end{equation}
and 
\begin{equation}
\left\| \sum_{j> (\log(1/\epsilon))^{1+\delta}} j\mu\tilde{B}^j \tilde{r}^{(1)}\,\right\| = O(\epsilon^2) \label{eq:3sec3new}
\end{equation}
as $\epsilon\downarrow0.$

On the other hand, condition $i)$ (and an argument identical to that used to establish Proposition 2) guarantees that there exists $v<\infty$ such that 
\begin{equation}
\sup_{1\le m\le (\log(1/\epsilon))^{1+\delta}} \left\| B_1(\epsilon)\cdots B_m(\epsilon) - \left(\tilde{B}^m + \epsilon \sum_{i=1}^{m-1} i \tilde{B}^i \tilde{B}^{(1)}\tilde{B}^{m-i-1}\right)    \right\| = O(\epsilon^2(\log(1/\epsilon))^v) \label{eq:4sec3new}
\end{equation}
as $\epsilon\downarrow0$. In addition, condition $i)$ also ensures that 
\begin{equation}
\sup_{1\le m\le (\log(1/\epsilon))^{1+\delta}} \| r_m(\epsilon)-(\tilde{r} + \epsilon(m-1)\tilde{r}^{(1)})  \| = O(\epsilon^2(\log(1/\epsilon))^s) \label{eq:5sec3new}
\end{equation}
as $\epsilon\downarrow 0$. Now (\ref{eq:4sec3new}) and (\ref{eq:5sec3new}) guarantee that 
\begin{align}
\left\| \sum_{j\le (\log(1/\epsilon))^{1+\delta}} \mu\left(\prod_{k=1}^{j}B_k(\epsilon)\right)r_{j+1}(\epsilon) - \sum_{j\le (\log(1/\epsilon))^{1+\delta}} \mu\left(\tilde{B}^j + \epsilon\sum_{i=1}^{j-1}i\tilde{B}^i \tilde{B}^{(1)}\tilde{B}^{m-i}\right)\tilde{r} \right. \nonumber \\
\left.  -\,\epsilon\sum_{j\le (\log(1/\epsilon))^{1+\delta}} \mu \tilde{B}^j j \tilde{r}^{(1)} \right\| = O(\epsilon^2(\log(1/\epsilon))^w) \label{eq:6sec3new}
\end{align}
as $\epsilon\downarrow0$, where $w = \max(s,v+1+\delta)$. Combining (\ref{eq:6sec3new}) with (\ref{eq:1sec3new}), (\ref{eq:2sec3new}), and (\ref{eq:3sec3new}), we conclude that
\begin{align}
&\left\|\sum_{j=0}^{\infty} \mu\left(\prod_{k=1}^{j}B_k(\epsilon) \right)r_{j+1}(\epsilon) - \sum_{j=0}^{\infty}\mu\tilde{B}^j \tilde{r} - \epsilon \sum_{j=0}^{\infty}\sum_{i=1}^{j-1}i\mu\tilde{B}^i \tilde{B}^{(1)}\tilde{B}^{j-i}\tilde{r} - \epsilon\sum_{j=0}^{\infty} j\mu\tilde{B}^j\tilde{r}^{(1)} \right\|\nonumber\\
&\quad= O(\epsilon^2(\log(1/\epsilon))^w) \label{eq:7sec3new}
\end{align}
as $\epsilon\downarrow0.$ We now use Proposition 1 to simplify the sums in (\ref{eq:7sec3new}), thereby yielding the theorem. 
\end{proof}

Suppose now that $P_j(\epsilon) = P((j-1)\epsilon)$ for $j\ge1$ and $\epsilon>0$. In this case, Theorem 3 holds when $P(\cdot)$ is differentiable in a neighborhood of 0, and
\[
\sup_{\theta\ge0} \|B^l(\theta)\| < 1,
\]
where $B(\theta)$ is the appropriate principal sub-matrix of $P(\theta)$. In this setting, $\tilde{B} = B(0)$, $\tilde{B}^{(1)} = B^{(1)}(0)$, and $\tilde{r}$ and $\tilde{r}^{(1)}$ have entries given by $\tilde{r}(x) = r(x) + \sum_{y\in C^c} P(0,x,y)r(y)$, $\tilde{r}^{(1)}(x) =\sum_{y\in C^c} P^{(1)}(0,x,y)r(y).$ While Theorem 3 can be stated directly in terms of $P(\cdot)$, we choose to use the hypotheses of Theorem 3 to make clear that the slow variation of the $B_i$'s is only required for the first $(\log(1/\epsilon))^{1+\delta}$ transitions. 

We follow the same approach as in Section 2 to apply this approximation to models with a given sequence $(P_j: j\ge1)$ of transition matrices. Assuming that 
\[
\max_{1\le j\le b/(1-\beta)} \| P_j - P_1 - (j-1)(P_2-P_1)\|
\]
is small for $b$ large, Theorem 3 should provide a good approximation to $\delta$, with $\tilde{B} = B_1$, $\tilde{r} = r_1$, $\epsilon\tilde{B}^{(1)} = (B_2-B_1)$, and $\epsilon\tilde{r}^{(1)}$ having entries given by $\epsilon\tilde{r}^{(1)}(x) = \sum_{y\in C^c} (B_2(x,y)-B_1(x,y)) r(y)$. As in Section 2, we can also develop a second order correction for $\delta$ that reflects the ``curvature" in the sequence $(P_j:j\ge 1)$; we omit the details.

\section{Approximating ${\E \lowercase{r}(X_{\lowercase{n}})}$}
We turn next to the question of how to approximate the transient quantity $\E r(X_n)$, when $X$ is a Markov chain with slowly changing transition probabilities. In particular, given the sequence $P_1,P_2,\ldots$ of transition matrices, $\chi_n \triangleq \E r(X_n)$ can be expressed as
\[
\chi_n = \mu P_1 P_2\cdots P_n r.
\]
We develop two different approximations in this setting. The first is appropriate when $n$ is small, while the second requires that $n$ be large. Once again, we consider a parameterized family of transition matrices $(P_i(\epsilon):i\ge 1, \epsilon>0)$. 
\begin{theorem}\label{thm:4}
	Assume that $n=n(\epsilon)$ is such that $n/\log(1/\epsilon)\rightarrow\infty$ and $n = o(\epsilon^{-1/3})$ as $\epsilon\downarrow0$. Suppose that there exist matrices $\tilde{P}$, $\tilde{P}^{(1)}$ for which 
	\begin{equation}
	\sup_{1\le i \le n} \| P_i(\epsilon) - (\tilde{P}+\epsilon(i-1)\tilde{P}^{(1)}) \| = O(\epsilon^{2}n^2) \label{eq:1sec4new}
	\end{equation}
	as $\epsilon\downarrow0$, where $\tilde{P}$ is aperiodic and irreducible. Then,
	\[
	\mu\prod_{j=1}^{n}P_j(\epsilon)r = \tilde{\pi} r + \epsilon n \tilde{\pi}\tilde{{P}}^{(1)}(I-\tilde{P}+\tilde{\Pi})^{-1}r - \epsilon \tilde{\pi}\tilde{P}^{(1)}(I-\tilde{P}+\tilde{\Pi})^{-2}r + O(\epsilon^{2}n^3)
	\]
	as $\epsilon\downarrow 0$, where $\tilde{\pi}$ is the row vector corresponding to the stationary distribution of $\tilde{P}$, and $\tilde{\Pi}$ is the rank one matrix with all rows identical to $\tilde{\pi}$. 
\end{theorem}

Theorem \ref{thm:4} states that if we have slow variation of the $P_i(\epsilon)$'s over $[0,n]$, then we have an approximation to $\mu\prod_{j=1}^{n}P_j(\epsilon)r$ with an error of order $O(\epsilon^{2}n^3)$ as $\epsilon\downarrow0.$ Note that the error term is of smaller order than the two asymptotic corrections of orders $n\epsilon$ and $\epsilon$ when $n = o(\epsilon^{-1/3})$. We further note that the asymptotic corrections involve the \textit{fundamental matrix} $(I-\tilde{P} + \tilde{\Pi})^{-1} = \sum_{n=0}^{\infty}(\tilde{P}-\tilde{\Pi})^n = I+\sum_{n=1}^{\infty}(\tilde{P}^n - \tilde{\Pi})$; see \cite{kemeny1960finite} for a discussion of its role in the analysis of Markov chains with stationary transition probabilities. 

The hypotheses of Theorem \ref{thm:4} hold when $P_i(\epsilon) = P((i-1)\epsilon)$ for $i\ge 1$, where $P(\cdot)$ is twice continuously differentiable in a neighborhood of $0$, with $P(0)$ aperiodic and irreducible. In our proof, we exploit the fact that $\tilde{\Pi}\tilde{P} = \tilde{\Pi}$ and that $A\tilde{\Pi} = \tilde{\Pi}$ whenever $A$ is stochastic. 

\begin{proof}{Proof of Theorem \ref{thm:4}.}
We note that (\ref{eq:1sec4new}) implies that for $\epsilon$ sufficiently small, there exists $d<\infty$ such that 
\[
\|	P_i(\epsilon) -(\tilde{P}+\epsilon(i-1)\tilde{P}^{(1)})\| \le d \,\epsilon^{2}n^2
\]
for $1\le i\le n$. With $f_m$ describing the same quantity as in the proof of Proposition 2, we now apply inequality (\ref{eq:A5}) to conclude that
\begin{align*}
f_{m+1} & \le f_m + d\,\epsilon^{2} n^2 + \epsilon f_m m \|\tilde{P}^{(1)}\| + \epsilon^2 m\sum_{i=1}^{m}i\|\tilde{P}^{(1)}\|^2 \\
&\le f_m(1+\epsilon m \|\tilde{P}^{(1)}\|) + d\,\epsilon^{2}n^2 + \epsilon^2 n^3\|\tilde{P}^{(1)}\|^2 \\
&	\le f_m(1+\epsilon n\|\tilde{P}^{(1)}\|) + d''\epsilon^{2}n^3
\end{align*}
for $1\le m+1\le n$ and for some constant $d''$. The proof of Proposition 2 then shows that 
\[
f_m\le 2d'' \epsilon^2n^3(1+\epsilon n\|\tilde{P}^{(1)}\|)^{n}
\]
for $1\le m\le n$. Because $n = o(\epsilon^{-1/2})$, $n \log(1+\epsilon n\|\tilde{P}^{(1)}\|)\rightarrow 0$ as $\epsilon\downarrow 0$, and hence 
\[
f_m = O(\epsilon^{2}n^3)
\]
as $\epsilon\downarrow 0$ and uniformly in $[1,n]$, proving that 
\begin{equation}
\mu\prod_{j=1}^{n}P_j(\epsilon) r = \mu \tilde{P}^n r +\epsilon \mu\sum_{i=1}^{n-1}i \tilde{P}^{i}\tilde{P}^{(1)}\tilde{P}^{n-i-1} r + O(\epsilon^{2}n^3) \label{eq:286C}
\end{equation}
as $\epsilon\downarrow0.$ Observe, the second term equals 
\begin{align}
&\epsilon\mu  \sum_{i=1}^{n-1}i(\tilde{P}-\tilde{\Pi})^{i}\tilde{P}^{(1)}\tilde{P}^{n-i-1}r + \epsilon\mu \sum_{i=1}^{n-1}i\tilde{\Pi} \tilde{P}^{(1)}\tilde{P}^{n-i-1}r \nonumber\\
=\, & \epsilon\mu  \sum_{i=1}^{n-1}i(\tilde{P}-\tilde{\Pi})^{i}\tilde{P}^{(1)}(\tilde{P}-\tilde{\Pi})^{n-i-1}r +\epsilon\mu  \sum_{i=1}^{n-2}i(\tilde{P}-\tilde{\Pi})^{i-1}\tilde{P}^{(1)}\tilde{\Pi}^{n-i-1}r+ \epsilon \sum_{i=1}^{n-1}i\tilde{\pi} \tilde{P}^{(1)}\tilde{P}^{n-i-1}r \label{eq:threeterm}
\end{align}
Next we show that $\tilde{P}^{(1)}\tilde{\Pi} = 0$. Divide (\ref{eq:1sec4new}) through by $\epsilon$ and note that
\[
\left\| \frac{P_i(\epsilon) - \tilde{P}}{\epsilon} - (i-1)\tilde{P}^{(1)} \right\| \rightarrow 0
\] 
as $\epsilon\downarrow 0$, due to the fact that $n=o(\epsilon^{-1/2})$. Let $e$ be a column vector consisting all 1's. Note that $(P_i(\epsilon) - \tilde{P})e = 0$ since both matrices $P_i(\epsilon)$ and $\tilde{P}$ are stochastic. Therefore $\tilde{P}^{(1)}e = 0$, implying $\tilde{P}^{(1)}\tilde{\Pi} = 0$ (since $\tilde{\Pi}$ has identical entries in each column). It follows that the second term on the right-hand side of (\ref{eq:threeterm}) vanishes.

The aperiodicity of $\tilde{P}$ ensures that $\| (\tilde{P}-\tilde{\Pi})^k \| = \|\tilde{P}^k - \tilde{\Pi}\|\rightarrow 0$ geometrically fast in $k$. In view of the fact that $n/\log(1/\epsilon)\rightarrow\infty$ as $\epsilon\downarrow 0$, this implies that
\[
\|(\tilde{P}-\tilde{\Pi})^{\frac{n}{2}}\| = O(\epsilon^k)
\]
as $\epsilon\downarrow0$, for each $k\ge 1$. Hence, $\mu\tilde{P}^n r = \tilde{\pi}r + O(\epsilon^k)$ and
\begin{align}
&\left\|\mu \sum_{i=1}^{n-1} i (\tilde{P}-\tilde{\Pi})^{i}\tilde{P}^{(1)}(\tilde{P}-\tilde{\Pi})^{n-i-1} r\right\| \nonumber\\
\le\, &\sup_{j\ge \frac{n}{2}-1}\| (\tilde{P} - \tilde{\Pi})^j \| \cdot \left( \left\| \sum_{i \le \frac{n}{2}} i (\tilde{P}-\tilde{\Pi})^{i}\tilde{P}^{(1)} \right\| \cdot \|r\| + \left\| \sum_{\frac{n}{2}\le i< n} i \tilde{P}^{(1)}(\tilde{P} - \tilde{\Pi})^{n-i-1} \right\| \cdot \|r\| \right) \nonumber\\
=\,&  O\left(\sup_{j\ge \frac{n}{2}-1} \|(\tilde{P}-\tilde{\Pi})^j\| \right) = O(\epsilon^k) \label{eq:286A}
\end{align}
as $\epsilon\downarrow 0$, for each $k\ge 1$. 

Finally, the third term on the right-hand side of (\ref{eq:threeterm}) equals
\begin{align}
& \epsilon\sum_{i=1}^{n-1}i\tilde{\pi} \tilde{P}^{(1)}(\tilde{P}-\tilde{\Pi})^{n-i-1}r \nonumber\\
=\, &\epsilon\sum_{j=0}^{n-2}(n-1-j)\tilde{\pi} \tilde{P}^{(1)}(\tilde{P}-\tilde{\Pi})^{j} r \nonumber\\
=\,& \epsilon n \tilde{\pi} \tilde{P}^{(1)} (I-\tilde{P}+\tilde{\Pi})^{-1}  r + O(\epsilon^k) - \epsilon\sum_{j=0}^{n-2}(j+1) \tilde{\pi} \tilde{P}^{(1)}(\tilde{P}-\tilde{\Pi})^j  r \nonumber \\
=\, & \epsilon n \tilde{\pi} \tilde{P}^{(1)} (I-\tilde{P}+\tilde{\Pi})^{-1}  r - \epsilon \tilde{\pi}\tilde{P}^{(1)}(I-\tilde{P}+\tilde{\Pi})^{-2}r + O(\epsilon^k) \label{eq:286B}
\end{align}
as $\epsilon\downarrow 0$, for each $k\ge 1$. Note, the first step uses part $ii)$ of Proposition \ref{prop:1}. Combining (\ref{eq:286A}), (\ref{eq:286B}), (\ref{eq:286C}), and the fact that the second term vanishes, yields the theorem.  
\end{proof}

We turn next to an approximation that is appropriate for larger value of $n$. While the first approximation effectively ``Taylor expands" in terms of $P_1 = P_1(\epsilon)$, the second ``Taylor expands" in terms of $P_n = P_n(\epsilon)$. 

\begin{theorem}\label{thm:5}
	Suppose that $n=n(\epsilon)$ is such that $n/\log(1/\epsilon)\rightarrow\infty$ as $\epsilon\downarrow 0$. For $a(\epsilon)\rightarrow\infty$ as $\epsilon\downarrow 0$, let $m = m(\epsilon) = \lfloor \min(n/2, a(\epsilon)\log(1/\epsilon))\rfloor$. Assume there exist matrices $\tilde{P}$, $\tilde{P}^{(1)}$ for which 
	\[
	\sup_{0\le k \le m} \| P_{n-k}(\epsilon) - (\tilde{P}-\epsilon k\tilde{P}^{(1)}) \| = O(\epsilon^2 m^2)
	\]
	as $\epsilon\downarrow 0$, where $\tilde{P}$ is irreducible and aperiodic. Then, 
	\[
	\mu\prod_{k=1}^{n}P_k(\epsilon)r = \tilde{\pi}r  - \epsilon\tilde{\pi}\tilde{P}^{(1)} (I-\tilde{P}+\tilde{\Pi})^{-2}\tilde{P}r + O(\epsilon^2 m^3)
	\]
	as $\epsilon\downarrow0$, where $\tilde{\pi}$ is the stationary distribution of $\tilde{P}$ and $\tilde{\Pi}$ is the rank one matrix having rows identical to $\tilde{\pi}$.
\end{theorem}

\begin{proof}{Proof of Theorem \ref{thm:5}.}
Proposition \ref{prop:2} implies that 
\begin{equation}
\left\|  \prod_{0\le k \le m} P_{n-k}(\epsilon) - \left(\tilde{{P}}^{m+1}  -\epsilon\sum_{j=1}^{m} j \tilde{P}^{m-j}\tilde{P}^{(1)} \tilde{P}^{j} \right)   \right\| = O(\epsilon^2 m^3) \label{eq:A287}
\end{equation}
as $\epsilon\downarrow0$. As in the proof of Theorem 4, we find that the aperiodicity and irreducibility of $\tilde{P}$ imply that $\| \tilde{P}^m - \tilde{\Pi} \| = O(\epsilon^k)$ as $\epsilon\downarrow 0$, for each $k\ge 1$. Similarly, 
\begin{align}
O(\epsilon^k) &= \left\| \sum_{1\le j\le m/2} j\tilde{P}^{m-j}\tilde{P}^{(1)}\tilde{P}^j - \sum_{1\le j\le m/2} j \tilde{\Pi}\tilde{P}^{(1)} \tilde{P}^j          \right\|   \nonumber\\
&= \left\| \sum_{1\le j\le m/2} j\tilde{\Pi}\tilde{P}^{(1)}\tilde{P}^j - \sum_{1\le j\le m/2} j \tilde{\Pi}\tilde{P}^{(1)} (\tilde{P}-\tilde{\Pi})^j          \right\|   \label{eq:B287}
\end{align}
as $\epsilon\downarrow 0$, for each $k\ge 1$, and 
\begin{align}
O(\epsilon^k) &= \left\| \sum_{1\le j\le m/2} j \tilde{\Pi}\tilde{P}^{(1)} (\tilde{P}-\tilde{\Pi})^j          -\sum_{j\ge 1} j\tilde{\Pi}\tilde{P}^{(1)} (\tilde{P}-\tilde{\Pi})^j     \right\|   \nonumber\\
&= \left\| \sum_{1\le j\le m/2} j \tilde{\Pi}\tilde{P}^{(1)} (\tilde{P}-\tilde{\Pi})^j -  \tilde{\Pi}\tilde{P}^{(1)} (I-\tilde{P}+\tilde{\Pi})^{-2}\tilde{P}       \right\|   \label{eq:C287}
\end{align}
as $\epsilon\downarrow0$, for each $k\ge 1$, where part $ii)$ of Proposition \ref{prop:1} was used in the last line. Relations (\ref{eq:A287}), (\ref{eq:B287}), and (\ref{eq:C287}) imply that 
\[
\left\|  \prod_{0\le k \le m} P_{n-k}(\epsilon) - \tilde{\Pi}+ \epsilon\tilde{\Pi}\tilde{P}^{(1)} (I-\tilde{P}+\tilde{\Pi})^{-2}\tilde{P}   \right\| = O(\epsilon^2 m^3) 
\]
as $\epsilon\downarrow0$. Because $A\tilde{\Pi} = \tilde{\Pi}$ for any stochastic matrix $A$, it follows that
\begin{align*}
O(\epsilon^2 m^3)  &= \left\| \prod_{m<k< n} P_{n-k}(\epsilon)  \left(\prod_{0\le k \le m} P_{n-k}(\epsilon) - \tilde{\Pi}+ \epsilon\tilde{\Pi}\tilde{P}^{(1)} (I-\tilde{P}+\tilde{\Pi})^{-2}\tilde{P}  \right) \right\|\\
&= \left\| \prod_{k=1}^{n}P_k(\epsilon) - \tilde{\Pi}+ \epsilon\tilde{\Pi } \tilde{P}^{(1)}(I-\tilde{P}+\tilde{\Pi})^{-2}\tilde{P} \right\|,
\end{align*}
proving the theorem. 
\end{proof}

We note that when $n/(\log(1+\epsilon))^{1+\delta}\rightarrow\infty$ as $\epsilon\downarrow 0$ for some $\delta > 0$, we can always choose $m = (\log(1+\epsilon))^{1+\delta}$ ensuring that our error term $O(\epsilon^2 m^3)$ is of smaller order than these correction terms of order $\epsilon$.  

Theorem 5 makes no assumptions whatsoever on the $P_k(\epsilon)$'s for $k$ outside a ``logarithmic neighborhood" of time epoch $n$ (outside $[n-a(\epsilon)\log(1/\epsilon),n]$), and $n$ can grow arbitrarily rapidly as a function of $\epsilon$. In particular, the assumptions of Theorem 5 hold when $P_{n-k}(\epsilon) = P(1-k\epsilon)$, where $(P(\theta):-\infty<\theta<\infty)$ is such that $P(\cdot)$ is twice continuously differentiable in a neighborhood of $1$, with $P(1)$ aperiodic and irreducible. 

Given a family $(P_j: j\ge 1)$ of transition matrices, Theorem 5 suggests approximating $\chi_n$ via 
\[
\pi_n r - \pi_n \left(\frac{P_n - P_{n-j}}{j}\right)(I-P_n+\Pi_n)^{-2}P_nr,
\]
for some user-defined choice of difference increment $j\ge 1$, where $\pi_n$ is the stationary distribution of $P_n$ (assumed irreducible and aperiodic), and $\Pi_n$ is the rank one matrix having identical rows equal to $\pi_n$. We note that the first-order correction  $\chi_{n1} \triangleq -\epsilon\pi_n P_n^{(1)}(I-P_n+\Pi_n)^{-2} P_n r $ can be computed by solving the linear system 
\begin{align}
(I-P_n+\Pi_n)h_{n1} & =P_nr, \nonumber\\
(I-P_n+\Pi_n)h_{n2} & =h_{n1}, \label{eq:24}
\end{align}
and setting $\chi_{n1} = -\pi_n P_n^{(1)}h_{n2}$. The fact that the coefficient matrices for these two linear systems are identical simplifies both numerical and closed form computation. 

By pre-multiplying (\ref{eq:24}) by $\pi_n$, we conclude that $\pi_n r = \pi_n h_{n1} = \pi_n h_{n2}$. So, we may re-write (\ref{eq:24}) as 
\begin{align}
(I-P_n)h_{n1} & =P_nr - \pi_n r e, \nonumber\\
(I-P_n)h_{n2} & =h_{n1} - \pi_n h_{n1} e, \label{eq:25}
\end{align}
where $e=(1,\ldots,1)^\top$ is the column vector consisting $1$'s. We recognize (\ref{eq:25}) as two \textit{Poisson equations} for the Markov chain having stationary transition matrices $P_j = P_n$ for $j\ge 1$. This is a discrete-time function analog to the first order term in the uniform acceleration (UA) asymptotic deduced by \cite{massey1998uniform}. (Their result is obtained for finite state Markov jump processes and focuses on the probability mass function version of Poisson's equation, where the unknown appears as a row vector pre-multiplying $(I-P_n)$, rather than the function version of Poisson's equation.) In addition to extending the theory to discrete time, our result makes clear that the approximation applies in much greater generality than the previous literature suggests. In particular, this first-order refinement holds whenever $n$ is large and $\epsilon$ is small, with no serious restriction on how large $n$ must be relative to $1/\epsilon$ (other than the very mild requirement that $n$ be large relative to $(\log(1/\epsilon))$). In contrast, we note that the uniform acceleration asymptotic relies on a time scaling of order $1/\epsilon$ in its derivation. In addition, our argument makes clear that only the transition matrices in a logarithmic neighborhood of the time $n$ under consideration play a role in the validity of the approximation.

We now provide a second-order refinement for $\chi_n$, based on ``Taylor expanding" in terms of $P_n = P_n(\epsilon)$. As with the first order refinement, it corresponds to a discrete time analog to the second order term in the uniform acceleration asymptotic expansion due to \cite{massey1998uniform}. (The proof of Theorem \ref{thm:6}	 is omitted, given the similarity to that of Theorem \ref{thm:5}.)
\begin{theorem}\label{thm:6}
	Suppose that $n=n(\epsilon)$ is such that $n/(\log(1/\epsilon))\rightarrow\infty$, and let $m=m(\epsilon)$ be defined as in Theorem 5. Assume there exist matrices $\tilde{P}$, $\tilde{P}^{(1)}$, and $\tilde{P}^{(2)}$ for which 
	\[
	\sup_{0\le k \le m} \left\|  P_{n-k}(\epsilon) - \left(\tilde{P} - \epsilon k \tilde{P}^{(1)} + \frac{\epsilon^2}{2}k^2\tilde{P}^{(2)}\right)  \right\| = O(\epsilon^3 m^3)
	\]
	as $\epsilon\downarrow0$, with $\tilde{P}$ irreducible and aperiodic. Then,
	\begin{align*}
	\mu \prod_{k=1}^{n}P_k(\epsilon)r =\,& \tilde{\pi} r - \epsilon\tilde{\pi} \tilde{P}^{(1)}(I-\tilde{P}+\tilde{\Pi})^{-1}\tilde{P}r + \epsilon^2 \tilde{\pi} \tilde{P}^{(1)} (I-\tilde{P}+\tilde{\Pi})^{-2} \tilde{P}^{(1)}(I-\tilde{P}+\tilde{\Pi})^{-2}\tilde{P} r\\
	&+ \epsilon^2 \left(\frac{1}{2}\tilde{\pi} \tilde{P}^{(2)} + \tilde{\pi} \tilde{P}^{(1)}(I-\tilde{P}+\tilde{\Pi})^{-1}\tilde{P}^{(1)}\right) \cdot \left[2(I-\tilde{P}+\tilde{\Pi})^{-3} - (I-\tilde{P}+\tilde{\Pi})^{-2} \right]\tilde{P}r\\
	& + O(\epsilon^3 m^4)
	\end{align*}
	as $\epsilon\downarrow0$, where $\tilde{\pi}$ is the stationary distribution of $\tilde{P}$ and $\tilde{\Pi}$ is the rank one matrix having rows identical to $\tilde{\pi}$.
\end{theorem}

As with Theorem 5, one possible choice for $m$ is $m=(\log(1/\epsilon))^{1+\delta}$ for some $\delta >0$, in which case the error term $O(\epsilon^3 m^4)$ is of smaller order than the correction term of order $\epsilon^2$. 

We close this section by noting that our arguments generalize (with essentially no changes to the proofs) to continuous state space Markov chains, provided that we suitably generalize the norms that are used. In particular, if $S$ is a general state space, we use the definitions 
\begin{align*}
&\|\eta\| = \sup_{B\subset S} |\eta(B)|,\\
&\|A\| = \sup_{x\in S} \|A(x,\cdot)\|,\\
&\|f\| = \sup_{x\in S} |f(x)|,
\end{align*}
for finite (signed) measures $\eta$, kernels $A$ (so that $A(x,\cdot)$ is a finite (signed) measure for each $x\in S$), and functions $f$. (Strictly speaking, our supremum over $B$ is over measurable subsets of $S$, and we require that $A$ and $f$ be suitably measurable.)

With these definitions in hand, Theorem 5 (for example) generalizes as follows: 
\begin{theorem}
	Suppose that $n/\log(1/\epsilon)\rightarrow\infty$ as $\epsilon\downarrow0$, and let $m$ be defined as in Theorem 5. Assume there exist kernels $\tilde{P}$ and $\tilde{P}^{(1)}$ for which
	\[
	\sup_{0\le k \le m} \| P_{n-k}(\epsilon) - (\tilde{P}-\epsilon k\tilde{P}^{(1)}) \| = O(\epsilon^2 m^2)
	\]
	as $\epsilon\downarrow 0$, where $\tilde{P}$ has a stationary distribution $\tilde{\pi}$. Let $\tilde{\Pi}$ be the kernel for which $\tilde{\Pi}(x,dy) = \tilde{\pi}(dy)$ for each $x,y\in S$, and suppose there exists $l\ge 1$ such that 
	\begin{equation}
	\|\tilde{P}^l - \tilde{\Pi}\| < 1. \label{eq:matrixassm}
	\end{equation}
	Then, $(I-\tilde{P}+\tilde{\Pi})$ has an inverse on the space of bounded (measurable) functions on $S$, and 
	\[
	\mu \prod_{k=1}^{n} P_k(\epsilon)r = \tilde{\pi} r -\epsilon \tilde{\pi}\tilde{P}^{(1)} (I-\tilde{P}+\tilde{\Pi})^{-2} \tilde{P} r + O(\epsilon^2 m^3)
	\]
	as $\epsilon\downarrow0$, provided $\|r\|<\infty.$
\end{theorem}

The assumption (\ref{eq:matrixassm}) on the transition kernel $\tilde{P}$ is identical to assuming that $\tilde{P}$ is aperiodic and uniformly ergodic (or, equivalently, that $\tilde{P}$ is \textit{Doeblin}; see \cite{doob1953stochastic}).

\section{Approximating Cumulative Reward}
In this section, we develop an approximation for 
\[
\tau_n = \E \sum_{j=0}^{n-1} r(X_j),
\]
when $X$ is slowly changing. Our approximation relies on the first approximation of Section 4, in which $\E r(X_j)$ is approximated by ``Taylor expanding" in terms of the $P_1$ dynamics. 

\begin{theorem}\label{thm:8}
	Suppose that $n=n(\epsilon)\rightarrow\infty$, so that $n/\log(1/\epsilon)\rightarrow\infty$ and $n= o(\epsilon^{-1/4})$ as $\epsilon\downarrow0$. Assume that there exists $\delta>0$ and matrices $\tilde{P}$, $\tilde{P}^{(1)}$ for which
	\[
	\sup_{1\le i \le n} \| P_i(\epsilon) - (\tilde{P} + \epsilon(i-1)\tilde{P}^{(1)})\|  = O(\epsilon^2 n^2)
	\]
	as $\epsilon\downarrow 0$, where $\tilde{P}$ is aperiodic and irreducible. Then, 
	\begin{align*}
	\sum_{j=0}^{n-1}\mu\prod_{k=1}^{j}P_k(\epsilon)r =\, & (n-1)\tilde{\pi} r + \mu (I-\tilde{P}+\tilde{\Pi})^{-1} r + \epsilon \mu \tilde{P} (I-\tilde{P}+\tilde{\Pi})^{-2}\tilde{P}^{(1)}(I-\tilde{P}+\tilde{\Pi})^{-1} r \\
	& + \epsilon \frac{(n-1)(n-2)}{2}\tilde{\pi}\tilde{P}^{(1)}(I-\tilde{P}+\tilde{\Pi})^{-1} r -\epsilon (n-1) \tilde{\pi} \tilde{P}^{(1)} (I-\tilde{P}+\tilde{\Pi})^{-2}\tilde{P} r \\
	& + \epsilon \tilde{\pi} \tilde{P}^{(1)}(I-\tilde{P}+\tilde{\Pi})^{-3}\tilde{P} r + o(\epsilon)
	\end{align*}
	as $\epsilon\downarrow0$, where $\tilde{\pi}$ is the stationary distribution of $\tilde{P}$ and $\tilde{\Pi}$ is the rank one matrix having rows identical to $\tilde{\pi}$.
\end{theorem}
\begin{proof}{Proof of Theorem \ref{thm:8}.}
As in the proof of Theorem 4, we find that the hypotheses guarantee that \[
\mu \prod_{i=1}^{j}P_i(\epsilon) = \mu \tilde{P}^j r + \epsilon \mu\sum_{i=1}^{j-1}i \tilde{P}^i \tilde{P}^{(1)}\tilde{P}^{j-1-i} r + O(\epsilon^2j^3)
\]
as $\epsilon\downarrow 0$, uniformly in $1\le j\le n$. Consequently, 
\begin{equation}
\sum_{j=0}^{n-1}\mu \prod_{i=1}^{j}P_i(\epsilon)r = \mu \sum_{j=0}^{n-1}\tilde{P}^j r + \epsilon\mu \sum_{i=1}^{n-2}i\tilde{P}^i \tilde{P}^{(1)} \sum_{j=i+1}^{n-1}\tilde{P}^{j-i-1} r + O(\epsilon^2n^4) \label{eq:288A}
\end{equation}
as $\epsilon\downarrow 0$. Since $n/\log(1/\epsilon) \rightarrow\infty$ and $\|\tilde{P}^n -\tilde{\Pi}\|\rightarrow 0$ geometrically fast, it follows that 
\begin{equation}
\mu\sum_{j=0}^{n-1}\tilde{P}^j r = (n-1)\tilde{\pi} r + \mu (I-\tilde{P}+\tilde{\Pi})^{-1} r + O(\epsilon^k)\label{eq:288B}
\end{equation}
as $\epsilon\downarrow 0$, for each $k\ge 1$. The second term on the right-hand side of (\ref{eq:288A}) equals 
\begin{align}
\nonumber	&\epsilon \mu \sum_{i=1}^{n-2} i \tilde{P} (\tilde{P}-\tilde{\Pi})^{i-1}\tilde{P}^{(1)}\sum_{j=0}^{n-2-i}(\tilde{P}-\tilde{\Pi})^j r + \epsilon\mu \sum_{i=1}^{n-2}i \tilde{\Pi}\tilde{P}^{(1)} \sum_{j=0}^{n-2-i}(\tilde{P}-\tilde{\Pi})^j r \\
\nonumber	=\,& \epsilon \mu \tilde{P}(I-\tilde{P}+\tilde{\Pi})^{-2}\tilde{P}^{(1)}(I-\tilde{P}+\tilde{\Pi})^{-1}r + \epsilon\sum_{i=1}^{n-2}i \tilde{\pi}\tilde{P}^{(1)}\sum_{j=0}^{n-2}(\tilde{P}-\tilde{\Pi})^j r\\
\nonumber	& -\epsilon \sum_{j=1}^{n-2}\sum_{i=n-1-j}^{n-2}i \tilde{\pi}\tilde{P}^{(1)}(\tilde{P}-\tilde{\Pi})^j r + O(\epsilon^k)\\
\nonumber	=\, & \epsilon\mu \tilde{P}(I-\tilde{P}+\tilde{\Pi})^{-2} \tilde{P}^{(1)}(I-\tilde{P}+\tilde{\Pi})^{-1} r + \epsilon\frac{(n-1)(n-2)}{2}\tilde{\pi}\tilde{P}^{(1)}(I-\tilde{P}+\tilde{\Pi})^{-1}r \\
\nonumber	& - \frac{\epsilon}{2} \sum_{j=1}^{n-2}\left[(2n-3)j - j^2 \right] \tilde{\pi}\tilde{P}^{(1)}(\tilde{P}-\tilde{\Pi})^j r + O(\epsilon^k)\\
=\, \nonumber& \epsilon\mu \tilde{P}(I-\tilde{P}+\tilde{\Pi})^{-2} \tilde{P}^{(1)}(I-\tilde{P}+\tilde{\Pi})^{-1} r + \epsilon\frac{(n-1)(n-2)}{2}\tilde{\pi}\tilde{P}^{(1)}(I-\tilde{P}+\tilde{\Pi})^{-1}r \\
\nonumber& - \frac{\epsilon}{2} (2n-3) \tilde{\pi}\tilde{P}^{(1)}(I-\tilde{P}+\tilde{\Pi})^{-2} \tilde{P}r + \frac{\epsilon}{2}\tilde{\pi}\tilde{P}^{(1)}\left[2(I-\tilde{P}+\tilde{\Pi})^{-3}- (I-\tilde{P}+\tilde{\Pi})^{-2}\right]\tilde{P}r
\\	
& O(\epsilon^k) \nonumber\\
=\, \nonumber& \epsilon\mu \tilde{P}(I-\tilde{P}+\tilde{\Pi})^{-2} \tilde{P}^{(1)}(I-\tilde{P}+\tilde{\Pi})^{-1} r + \epsilon\frac{(n-1)(n-2)}{2}\tilde{\pi}\tilde{P}^{(1)}(I-\tilde{P}+\tilde{\Pi})^{-1}r \\
& - {\epsilon} (n-1) \tilde{\pi}\tilde{P}^{(1)}(I-\tilde{P}+\tilde{\Pi})^{-2} \tilde{P}r + {\epsilon}\tilde{\pi}\tilde{P}^{(1)}(I-\tilde{P}+\tilde{\Pi})^{-3}\tilde{P}r + O(\epsilon^k) \label{eq:288C}
\end{align}
as $\epsilon\downarrow 0$, for each $k\ge 1$. The theorem is proved by using relations (\ref{eq:288A}), (\ref{eq:288B}), and (\ref{eq:288C}), noting that our assumption that $n = o(\epsilon^{-1/4})$ implies that $O(\epsilon^2 n^4) = o(\epsilon)$ as $\epsilon\downarrow 0.$ 
\end{proof}

So, for moderate values of $n$ (of smaller order than $\epsilon^{-1/4}$), we can approximate the expected cumulative reward via the stationary dynamics of a Markov chain having transition matrix $P_1$. In particular, to obtain an approximation for a given sequence $(P_i: i \ge 1)$, we replace $\tilde{P}$ by $P_1$ and approximate $\epsilon\tilde{P}^{(1)}$ via a finite difference as in (\ref{eq:7anew}). 

Finally, we observe that the hypotheses of Theorem 7 hold when $P_i(\epsilon) = P((i-1)\epsilon)$ for $i\ge 1$, with $P(\cdot)$ twice continuously differentiable in a neighborhood of 0, assuming that $P(0)$ is irreducible and aperiodic.

\section{Extension to Markov Jump Processes}
The theory of Sections 2 through 5 extends in a straightforward fashion to finite-state continuous-time Markov jump processes. We illustrate this by generalizing Theorems 5 and 6 to this setting. 

Let $(Q(t):t \ge 0)$ be the family of rate matrices associated with the Markov process $X=(X(t):t\ge 0)$ having non-stationary transition probabilities, so that
\[
\E \left[ f(X(t+h)) | X(u):0\le u\le t \right] = f(X(t)) + (Q(t)f) (X(t)) h + o(h)
\]
a.s. as $h\downarrow 0$, for any $f: S\rightarrow\mathbb{R}$. Also, for $0\le s\le t$, let $P(s,t)$ be the square matrix having entries $P(s,t;x,y) = P(X(t)=y | X(s)=x)$ for $x,y\in S$, where $S$ is (as in the earlier sections of this paper) the state space of $X$. Then, $P(0,u+t) = P(0,u)P(u,u+t)$ for $u,t\ge 0$. Also, if 
\[
\lambda\triangleq \frac{1}{2}\sup_{u\le s\le u+t} \|Q(s)\|,
\]
then 
\begin{align*}
P(u,u+t) &= \sum_{n=0}^{\infty} e^{-\lambda t} \frac{(\lambda t)^n}{n!} \int_{0}^{t}\int_{u_1}^{t}\cdots \int_{u_{n-1}}^{t}R(t-u_n)\cdots R(t-u_1) du_n\cdots du_1 \frac{n!}{t^n}\\
&= \sum_{n=0}^{\infty} e^{-\lambda t}\frac{(\lambda t)^n}{n!}\E R(t(1-U_{(n)})\cdots R(t(1-U_{(1)})),
\end{align*}
where $(U_{(1)},\ldots,U_{(n)})$ are the order statistics from an independent and identically distributed sample of size $n$ from a uniform distribution on $[0,1]$, and $R(s)\triangleq \lambda^{-1}(\lambda I + Q(s))$ for $u\le s\le u+t$. This representation follows directly from the fact that $X$ can be ``uniformized" with respect to a Poisson process having rate $\lambda >0$, and the transition probabilities for $X$, conditional on a jump at time $s$, are given by the entries of the transition matrix $R(s)$; see \cite{massey1998uniform} for a further discussion. 

We now consider a parameterized setting in which we have a family $(Q(\epsilon;t):t\ge0,\epsilon>0)$ of rate matrices, with associated transition matrices $(P(\epsilon;s,t):0\le s\le t,\epsilon>0)$. 

\begin{theorem}\label{thm:9}
	For $\epsilon>0$, let $s=s(\epsilon)=(\log(1/\epsilon))^{1+\delta}$ for some $\delta>0$. Suppose that $t=t(\epsilon)$ is such that $t/s\rightarrow\infty$ as $\epsilon\downarrow 0$, and assume there exist matrices $\tilde{Q}$, $\tilde{Q}^{(1)}$ such that 
	\begin{equation}
	\sup_{0\le u \le s} \| Q(\epsilon;t-u) - (\tilde{Q} - \epsilon\tilde{Q}^{(1)}u) \|  = O(\epsilon^2 s^2) \label{eq:6.1}
	\end{equation}
	as $\epsilon\downarrow 0$, where $\tilde{Q}$ is an irreducible rate matrix. Then, 
	\[
	\mu P(\epsilon;0,t) r= \tilde{\pi} r - \epsilon\tilde{\pi}\tilde{Q}^{(1)}(\tilde{\Pi} - \tilde{Q})^{-2}r + O(\epsilon^2 s^3) 
	\]
	as $\epsilon\downarrow0$, where $\tilde{\pi}$ is the stationary distribution associated with $\tilde{Q}$, and $\tilde{\Pi}$ is the rank one matrix having all rows identical to $\tilde{\pi}$. 
\end{theorem}

\begin{proof}{Proof of Theorem \ref{thm:9}.}
Note that for some $\epsilon_0>0$, (\ref{eq:6.1}) ensures that 
\[
\sup_{0\le u\le s} \|Q(\epsilon;t-u)\| \le \frac{1}{2}\|\tilde{Q}\| + 1 \triangleq \tilde{\lambda},
\]
uniformly in $\epsilon\in (0,\epsilon_0)$. Then, set $R(\epsilon;u) = \tilde{\lambda}^{-1}(\tilde{\lambda}I + Q(\epsilon;u))$ for $t-s\le u\le t$. So, 
\begin{align}
\mu P(\epsilon;0,t) r &= \mu P(\epsilon;0,t-s)P(\epsilon;t-s,t)r \nonumber\\
& = \sum_{n=0}^{\infty} e^{-\tilde{\lambda}s} \frac{(\tilde{\lambda}s)^n}{n!} \mu P(\epsilon;0,t-s)\E R(\epsilon;t-sU_{(n)})\cdots R(\epsilon;t-sU_{(1)})r    \nonumber\\
&= \sum_{|\frac{n}{s}-\tilde{\lambda}|<\frac{\tilde{\lambda}}{4}} e^{-\tilde{\lambda}s}\frac{(\tilde{\lambda}s)^n}{n!}\mu P(\epsilon;0,t-s) \E R(\epsilon;t-s U_{(n)})\cdots R(\epsilon;t-sU_{(1)})r + O(\epsilon^k) \label{eq:6.2}
\end{align}
as $\epsilon\downarrow0$, for each $k\ge 1$ (since the probability that a Poisson random variable with mean $\tilde{\lambda}s$ lies more than $\tilde{\lambda}s/4$ from its mean is of the order of $\exp(-c(\log(1/\epsilon))^{1+\delta})$ for some $c>0$, using a standard large deviations tail bound; see \cite{dembo2010large}).

Assumption (\ref{eq:6.1}) guarantees that 
\begin{equation}
\sup_{0\le u \le s} \| R(\epsilon;t-u) - (\tilde{R}-\epsilon\tilde{R}^{(1)}u) \| = O(\epsilon^2 s^2) \label{eq:6.3}
\end{equation}
as $\epsilon\downarrow0$, where $\tilde{R} = \tilde{\lambda}^{-1}(\tilde{\lambda}I + \tilde{Q})$ and $\tilde{R}^{(1)} = \tilde{Q}^{(1)}/\tilde{\lambda}$. Furthermore, the stochastic matrix $\tilde{R}$ is irreducible and aperiodic (since it has positive diagonal entries because $\tilde{\lambda}> \|\tilde{Q}\|/2$). The bound (\ref{eq:6.4}), based on (\ref{eq:6.3}), establishes that 
\begin{align*}
& \left\| R(\epsilon;t-s u_{(n)})\cdots R(\epsilon;t-su_{(1)}) -\left(\tilde{R}^n - \epsilon\sum_{i=0}^{n-1}s u_{(i+1)}\tilde{R}^{n-i-1}\tilde{R}^{(1)}\tilde{R}^{i}\right)  \right\| \\
& \le \epsilon^2d''s^3(1+\epsilon c's)^n \le \epsilon^2 d'' s^3(1+\epsilon c's)^{\frac{5}{4}\tilde{\lambda}s}
\end{align*}
uniformly in $0\le u_{(1)}\le \cdots \le u_{(n)}\le 1$ for some constants $c'$, $d''$. It follows that 
\begin{align*}
\left\| R(\epsilon;t-s u_{(n)})\cdots R(\epsilon;t-su_{(1)}) -\left(\tilde{R}^n - \epsilon\sum_{i=0}^{n-1}s u_{(i+1)}\tilde{R}^{n-i-1}\tilde{R}^{(1)}\tilde{R}^{i}\right)  \right\| = O(\epsilon^2 s^3)
\end{align*}
as $\epsilon\downarrow0$, uniformly in $0\le u_{(1)}\le \cdots \le u_{(n)}\le 1$. Hence, 
\begin{align*}
\left\| \E R(\epsilon;t-s U_{(n)})\cdots R(\epsilon;t-sU_{(1)}) -\left(\tilde{R}^n - \epsilon\sum_{i=0}^{n-1}s \left(\frac{i+1}{n+1}\right)\tilde{R}^{n-i-1}\tilde{R}^{(1)}\tilde{R}^{i}\right)  \right\| = O(\epsilon^2 s^3)
\end{align*}
as $\epsilon\downarrow0$, uniformly in $n\in[\tilde{\lambda}s(3/4),\tilde{\lambda}s(5/4)]$, where we have used the fact that $\E U_{(i)} = i/(n+1)$ for $1\le i\le n$; see \cite{arnold2008first}. We now apply the same argument as in the proof of Theorem 5 (using the fact that $\|\tilde{R}^n - \tilde{\Pi}\|\rightarrow 0$ geometrically fast) to obtain
\begin{equation}
\left\| \E R(\epsilon;t-s U_{(n)})\cdots R(\epsilon;t-sU_{(1)})r -\left(\tilde{\pi} - \frac{\epsilon}{n+1}\tilde{\pi} \tilde{R}^{(1)}(I-\tilde{R}+\tilde{\Pi})^{-2} \right) r \right\| = O(\epsilon^2 s^3) \label{eq:6.5}
\end{equation}
as $\epsilon\downarrow 0$, uniformly in $n\in [\tilde{\lambda}s(3/4),\tilde{\lambda}s(5/4)]$. Plugging (\ref{eq:6.5}) into (\ref{eq:6.2}), we conclude that 
\begin{align*}
\mu P(\epsilon;0,t) &=  \sum_{|\frac{n}{s}-\tilde{\lambda}|<\frac{\tilde{\lambda}}{4}} e^{-\tilde{\lambda}s}\frac{(\tilde{\lambda}s)^n}{n!}\left[\tilde{\pi} r -  \frac{\epsilon}{n+1}\tilde{\pi} \tilde{R}^{(1)}(I-\tilde{R}+\tilde{\Pi})^{-2}  r\right] +  O(\epsilon^2 s^3)\\
&= \tilde{\pi} r - \frac{\epsilon}{\tilde{\lambda}}\tilde{\pi}\tilde{R}^{(1)}(I-\tilde{R}+\tilde{\Pi})^{-2} r +  O(\epsilon^2 s^3)\\
&= \tilde{\pi} r - \frac{\epsilon}{\tilde{\lambda}^2}\tilde{\pi}\tilde{Q}^{(1)}(I-\tilde{R}+\tilde{\Pi})^{-2} r +  O(\epsilon^2 s^3)\\
& = \tilde{\pi} r - {\epsilon}\tilde{\pi}\tilde{Q}^{(1)}(\tilde{\lambda}\tilde{\Pi}-\tilde{Q})^{-2} r +  O(\epsilon^2 s^3)
\end{align*}
as $\epsilon\downarrow0$. 

We now consider $W = \tilde{Q}^{(1)}(\tilde{\lambda}\tilde{\Pi}-\tilde{Q})^{-2}$, so that $W(\tilde{\lambda}\tilde{\Pi}-\tilde{Q})^2 = \tilde{Q}^{(1)}$. Since $\tilde{\Pi} = \tilde{Q}\tilde{\Pi} = 0$ and $\tilde{\Pi}^2 = \tilde{\Pi}$, we find that 
\begin{equation}
W(\tilde{\lambda}^2\tilde{\Pi}+\tilde{Q}^2) = \tilde{Q}^{(1)} \label{eq:6.8}.
\end{equation}
Since $\tilde{Q}^{(1)} = \tilde{\lambda}\tilde{R}^{(1)}$, it follows that $\tilde{Q}^{(1)}e =0$. Of course, $\tilde{Q}e = 0$ as well, so (\ref{eq:6.8}) implies that 
\[
\tilde{\lambda}^2W\tilde{\Pi} e = 0,
\]
which yields $\tilde{\lambda}^2We = 0$. But $\tilde{\lambda}\neq 0$, so we conclude that $We = 0$. Since $\tilde{\Pi} = e \tilde{\pi}$, (\ref{eq:6.8}) implies that 
\[
W\tilde{Q}^2 = \tilde{Q}^{(1)},
\]
and hence 
\[
W(\tilde{\Pi} - \tilde{Q})^2 = \tilde{Q}^{(1)}.
\]
It is well known that $\tilde{\Pi}-\tilde{Q}$ is non-singular when $\tilde{Q}$ is irreducible, so consequently $W = \tilde{Q}^{(1)}(\tilde{\Pi}-\tilde{Q})^{-2}$. Hence, $\tilde{\pi}\tilde{Q}^{(1)}(\tilde{\lambda}\tilde{\Pi}-\tilde{Q})^{-2}r = \tilde{\pi}\tilde{Q}^{(1)}(\tilde{\Pi}-\tilde{Q})^{-2}r,$ proving the theorem. 
\end{proof}

As in our discrete time theory, the slow variation assumption (\ref{eq:6.1}) is required only in a ``logarithmic neighborhood" of time $t$. No time re-scaling is required in order that this result be valid, nor do we need to assume that the supremum of $\|Q(\epsilon;\cdot)\|$ is bounded over $[0,t]$ uniformly in $\epsilon$. We further note that (\ref{eq:6.1}) is valid (for example) when $Q(\epsilon;u) = \utilde{Q}(\epsilon u) $ for $u\ge 0$, and $t(\epsilon) = t/\epsilon$ with $\utilde{Q}(\cdot)$ twice continuously differentiable in a neighborhood of $t$, where $\tilde{Q} = \utilde{Q}(t)$ and $\tilde{Q}^{(1)} = \utilde{Q}^{(1)}(t)$. This is precisely the ``uniform acceleration" (UA) asymptotic environment considered by \cite{massey1998uniform}. 

There is no intrinsic difficulty in extending to higher order approximations. As an illustration, we state the following result (without proof) for the second order approximation.

\begin{theorem}
	For $\epsilon>0$, let $s=s(\epsilon)=(\log(1/\epsilon))^{1+\delta}$ for some $\delta>0$. Suppose that $t=t(\epsilon)$ is such that $t/s\rightarrow\infty$ as $\epsilon\downarrow 0$, and assume there exist matrices $\tilde{Q}$, $\tilde{Q}^{(1)}$, and $\tilde{Q}^{(2)}$ such that 
	\begin{equation}
	\sup_{0\le u \le s} \left\| Q(\epsilon;t-u) - \left(\tilde{Q} - \epsilon\tilde{Q}^{(1)}u + \frac{\epsilon^2}{2}\tilde{Q}^{(2)}u^2 \right) \right\|  = O(\epsilon^3 s^3) \label{eq:6.19}
	\end{equation}
	as $\epsilon\downarrow 0$, where $\tilde{Q}$ is an irreducible rate matrix. Then, 
	\begin{align*}
	\mu P(\epsilon;0,t) r =\,& \tilde{\pi} r - \epsilon\tilde{\pi}\tilde{Q}^{(1)}(\tilde{\Pi} - \tilde{Q})^{-2}r + \frac{1}{2}\epsilon^2 \tilde{\pi}\tilde{Q}^{(2)}(\tilde{\Pi}-\tilde{Q})^{-3} r \\
	& + \epsilon^2 \tilde{\pi}\tilde{Q}^{(1)} \left( (\tilde{\Pi} - \tilde{Q})^{-2}\tilde{Q}^{(1)}(\tilde{\Pi} - \tilde{Q})^{-2} + 2 (\tilde{\Pi} - \tilde{Q})^{-1}\tilde{Q}^{(1)}(\tilde{\Pi} - \tilde{Q})^{-3} \right)r + O(\epsilon^3 s^4) 
	\end{align*}
	as $\epsilon\downarrow0$, where $\tilde{\pi}$ is the stationary distribution associated with $\tilde{Q}$, and $\tilde{\Pi}$ is the rank one matrix having all rows identical to $\tilde{\pi}$. 
\end{theorem}

\input{numerics}

\pdfbookmark[1]{References}{link7}

%
\end{document}

%% file: numerics.tex


\section{Numerical Example}
In this section, we illustrate our methodology by applying it to the problem of computing the infinite horizon discounted reward for a Markov chain $X = (X_n: n\ge 0)$ having non-stationary transition probabilities. Specifically, we consider an $(s,S)$ inventory model in which $r(x)=x$ and 

\begin{align}
\label{eq:9}
X_{k+1} =
\begin{cases}
X_k - D_k &~\text{if}~ X_k - D_k \ge s\\
S &~\text{if}~ X_k - D_k< S,
\end{cases}
\end{align}
for $k\ge0$. The non-stationarity arises as a consequence of time-varying demand. In particular, we assume that the $D_k$'s are independent Poisson random variables in which
\[
\E D_j = m + \epsilon j
\]
for some $j\ge 1$, where $m>0$ and $\epsilon\ge 0$. In our experiments, we chose $m=1$ and performed calculations at $(s,S,\alpha)\in \{ (4,10,0.1), (4,10,0.5),(4,10,1),(40,100,0.1),(40,100,0.5),(40,100,1) \}$. We then studied the quality of our approximations as a function of $\epsilon$. 

Recall that our approximations require the derivative matrices $\tilde{P}^{(1)}$ and $\tilde{P}^{(2)}$. In a typical application, the parameter $\epsilon$ is fixed, and these matrices need to be approximated via the finite differences given by (\ref{eq:7anew}) and (\ref{eq:7aa}). We choose $j = \lceil (1-e^{-\alpha})^{-1} \rceil$ in these formulae, and re-compute our finite difference approximations to $\tilde{P}^{(1)}$ and $\tilde{P}^{(2)}$ at each value of $\epsilon$. Since our finite difference approximations are non-linear in $\epsilon$, this implies that even our first-order approximation to $\kappa(\epsilon)$ is non-linear in $\epsilon$. 

In the current experiment, however, we also have the ability to compute the exact derivatives $\tilde{P}^{(1)}$ and $\tilde{P}^{(2)}$ via component-wise differentiation of the matrices $P_1(\epsilon)$ at $\epsilon = 0$. So, we compute two first-order approximations for $\kappa(\epsilon)$, one based on our finite-difference approximation for $\tilde{P}^{(1)}$ (denoted ``$1^{st}$ order FD") and the other based on the exact derivative (denoted ``$1^{st}$ order Exact"). Similarly, we compute two second-order approximations for $\kappa(\epsilon)$ (denoted ``$2^{nd}$ order FD" and ``$2^{nd}$ order Exact", respectively). 

Computing the exact value of $\kappa(\epsilon)$ is implemented by using (\ref{eq:2.2}), based on truncating the sum over $j$ at a suitable value $n-1$. We note that the norm of the tail sum of (\ref{eq:2.2}) is upper bounded via 
\[
\sum_{j=n}^{\infty} e^{-\alpha j} \|r\| = e^{-\alpha n} (1-e^{-\alpha})^{-1}S.
\]
We want the contribution of the tail sum to be small relative to the magnitude of the first and second order corrections to the stationary model in which $\epsilon = 0$, so we choose $n$ so that $e^{-\alpha n} (1-e^{-\alpha})^{-1}S \le \epsilon^6$ $\Big($ie. $n = \left\lceil-\frac{1}{\alpha}\log\left(\frac{\epsilon^6 (1-e^{-\alpha})}{S}\right)
\right\rceil$ $\Big)$. We denote the value of $\kappa(\epsilon)$ obtained through this truncation as ``Truncated True."

To test the quality of the truncation, we computed the exact derivatives $\kappa^{(1)}(0)$ and $\kappa^{(2)}(0)$ as determined by the coefficients in $\epsilon$ and $\epsilon^2/2$ appearing in (\ref{eq:secondterms}), and compared them to finite difference approximations to $\kappa^{(1)}(0)$ and $\kappa^{(2)}(0)$ as obtained from the ``Truncated True" approximations. The relative errors were uniformly under $0.1\%$ for $S=10$ and were under $5\%$ for $S=100.$

Tables 1 through 6 below provide the percent relative error (i.e., $100|\text{Approx} - \text{Truncated True}|/$ $\text{Truncated True}\,\%$) of our approximations, as a function of $\epsilon$, at each of our six combinations of $(s,S,\alpha)$. The tables are consistent with what we would expect from our approximations, in the sense that the relative error is smaller when $\epsilon$ is small, and typically also smaller when the second order approximation is used as compared to the first order approximation. In addition, as the discount rate $\alpha$ gets larger, the main contribution to the infinite horizon reward focuses to a greater degree on the early transitions at which the Taylor approximation around $P_1$ will be good. As expected, the tables do indeed show that the relative error typically decreases with larger discount rates. Note that there is no theoretical guarantee that using exact derivatives in our approximations will reduce the error relative to using finite difference approximations, and we find examples in the tables in which each dominates the other. In conclusion, it appears that this numerical investigation validates our approximations.

\begin{table}[h]
	\footnotesize
	\centering
	\caption{Relative Accuracy of the First and Second Order Approximations}
	\label{my-label}
	\begin{tabular}{|l|l|l|l|l|l|}
		\hline
		$\epsilon$ & Truncated True & $1^{st}$ Order FD \ & $1^{st}$ Order Exact \ & $2^{nd}$ Order FD \ & $2^{nd}$ Order Exact \ \\ \hline
		0          &      64.0915  &   0.0000                         &        0.0000                       &    0.0000                        &            0.0000                   \\
		0.001      &   64.0170  &   0.0002                     &     0.0000                          &   0.0002                         &   0.0000                   \\
		0.004      &   63.7936  &  0.0032                     &  0.0005                      &  0.0035             &   0.0000\\
		0.016      &   62.9040  & 0.0523                        &   0.0090                     &    0.0435           &    0.0040                           \\
		0.064      &   59.4335  &   0.8079                    & 0.1925                  &    0.1375            &   0.0334                          \\
		0.256      &   47.8842  &   10.8270                  &   6.0198                    &  34.3068            &   2.8580                            \\
		1.024      &   26.9824  &    72.7559                 &   145.4665               &  160.1515          &   55.6900                            \\ \hline
	\end{tabular}
\vspace{1ex}

 Parameters: $\alpha = 0.1$, $s =4$, $S= 10$
\end{table}
\begin{table}[h]
	\footnotesize
	\centering
	\caption{Relative Accuracy of the First and Second Order Approximations}
	\label{my-label}
	\begin{tabular}{|l|l|l|l|l|l|}
		\hline
		$\epsilon$ & Truncated True & $1^{st}$ Order FD \ & $1^{st}$ Order Exact \ & $2^{nd}$ Order FD \ & $2^{nd}$ Order Exact \ \\ \hline
		0          &      13.0039  &   0.0000                         &   0.0000                            &    0.0000                        &    0.0000                           \\
		0.001      &   13.0014  &   0.0000                        &     0.0000                          &   0.0000                         &   0.0000                            \\
		0.004      &   12.9936 &  0.0000                            & 0.0001                      &  0.0000                          &   0.0000                            \\
		0.016      &   12.9627  & 0.0006                    & 0.0019                         &    0.0013                 &   0.0000                            \\
		0.064      &   12.8415  &   0.0070                    & 0.0291                       &    0.02117               &  0.0035                             \\
		0.256      &   12.3875  &  0.0266 & 0.3892 & 0.2323                          &  0.1520                             \\
		1.024      &   10.9140  & 0.9337 &5.2123 &  4.5863                         &   4.6166                            \\ \hline
	\end{tabular}
\vspace{1ex}

Parameters: $\alpha = 0.5$, $s =4$, $S= 10$
\end{table}

\begin{table}[h]
	\footnotesize
	\centering
	\caption{Relative Accuracy of the First and Second Order Approximations}
	\label{my-label}
	\begin{tabular}{|l|l|l|l|l|l|}
		\hline
		$\epsilon$ & Truncated True & $1^{st}$ Order FD \ & $1^{st}$ Order Exact \ & $2^{nd}$ Order FD \ & $2^{nd}$ Order Exact \ \\ \hline
		0          &      5.8910  &   0.0000                         &       0.0000                        &    0.0000                        &    0.0000                           \\
		0.001      &   5.8906  &   0.0000                        &        0.0000                       &   0.0000                         &     0.0000                          \\
		0.004      &   5.8893 &  0.0000                            &      0.0000                         &  0.0000                          &    0.0000                           \\
		0.016      &   5.8843  & 0.0003                           &  0.0005                             &    0.0001       &   0.0000                            \\
		0.064      &   5.8648  &   0.0057                         &   0.0077                            &    0.0026       & 0.0007                              \\
		0.256      &   5.7904  &  0.0705 & 0.1041 & 0.0574      & 0.0320                              \\
		1.024      &   5.5316  & 0.6203 & 1.2102 &  0.9747       &  1.0713                             \\ \hline
	\end{tabular}
\vspace{1ex}

Parameters: $\alpha = 1.0$, $s =4$, $S= 10$
\end{table}

\begin{table}[h]
	\footnotesize
	\centering
	\caption{Relative Accuracy of the First and Second Order Approximations}
	\label{my-label}
	\begin{tabular}{|l|l|l|l|l|l|}
		\hline
		$\epsilon$ & Truncated True & $1^{st}$ Order FD \ & $1^{st}$ Order Exact \ & $2^{nd}$ Order FD \ & $2^{nd}$ Order Exact \ \\ \hline
		0          &      853.5824  &   0.0000                         &    0.0000                           &    0.0000                        &    0.0000                           \\
		0.001      &   852.8980  &   0.0008                       &   0.0008                            &   0.0000               &  0.0000                             \\
		0.004      &   850.9225 &  0.0122                            &  0.0124                             &  0.0011        &  0.0010                             \\
		0.016      &   843.9259  & 0.1637                           &   0.1665                            &    0.0524                  &   0.0494           \\
		0.064      &   823.6448  &   1.6917                         & 1.7372                  &    1.8612                 &  1.8027                          \\
		0.256      &   778.7003  &  12.3135 &13.1122 & 47.9461                          &  46.7974                             \\
		1.024      &   686.7012  & 62.0493 &78.7922 &  852.3577                         & 1008.1822                              \\ \hline
	\end{tabular}
	\vspace{1ex}
	
	Parameters: $\alpha = 0.1$, $s =40$, $S= 100$
\end{table}

\begin{table}[h]
	\footnotesize
	\centering
	\caption{Relative Accuracy of the First and Second Order Approximations}
	\label{my-label}
	\begin{tabular}{|l|l|l|l|l|l|}
		\hline
		$\epsilon$ & Truncated True & $1^{st}$ Order FD \ & $1^{st}$ Order Exact \ & $2^{nd}$ Order FD \ & $2^{nd}$ Order Exact \ \\ \hline
		0          &      151.2317  &   0.0000                         &    0.0000                           &    0.0000                        &       0.0000                       \\
		0.001      &   151.2256  &   0.0000                       &     0.0000                          &   0.0000                         &    0.0000                           \\
		0.004      &   151.2075 &  0.0000                            &   0.0000                            &  0.0000                          &   0.0000                            \\
		0.016      &   151.1350  & 0.0000                           &   0.0000                            &    0.0000                        &    0.0000                           \\
		0.064      &   150.8452  &   0.0000                         &   0.0000                            &    0.0000                        &    0.0000                           \\
		0.256      &   149.6945  &  0.0059 &0.0059 & 0.0059                          &   0.0059                            \\
		1.024      &   145.5617  & 0.3532 &0.3531 &  0.3530                         &0.3531                              \\ \hline
	\end{tabular}
\vspace{1ex}

Parameters: $\alpha = 0.5$, $s =40$, $S= 100$
\end{table}

\begin{table}[h]
	\footnotesize
	\centering
	\caption{Relative Accuracy of the First and Second Order Approximations}
	\label{my-label}
	\begin{tabular}{|l|l|l|l|l|l|}
		\hline
		$\epsilon$ & Truncated True & $1^{st}$ Order FD \ & $1^{st}$ Order Exact \ & $2^{nd}$ Order FD \ & $2^{nd}$ Order Exact \ \\ \hline
		0          &      58.2770  &   0.0000                         &     0.0000                          &    0.0000                        &    0.0000                           \\
		0.001      &   58.2764  &   0.0000                       &      0.0000                         &   0.0000                         &   0.0000                            \\
		0.004      &   58.2748 &  0.0000                            &   0.0000                            &  0.0000                          &  0.0000                             \\
		0.016      &   58.2684  & 0.0000                           &     0.0000                          &    0.0000                        &   0.0000                            \\
		0.064      &   58.2427  &   0.0000                         &    0.0000                           &    0.0000                        &  0.0000                             \\
		0.256      &   58.1398  &  0.0000 &0.0000 & 0.0000                         & 0.0000                              \\
		1.024      &   57.7292  & 0.0015 &0.0015 &  0.0016                         &0.0015                               \\ \hline
	\end{tabular}
\vspace{1ex}

Parameters: $\alpha = 1.0$, $s =40$, $S= 100$
\end{table}
\clearpage
\newpage